\documentclass[a4paper,11pt,DIV=11,% decrease borders by 2 (std 10 for 11pt)
abstract=on% return to old style centered abstract
]{scrartcl}
\usepackage{mathtools}
\mathtoolsset{showonlyrefs}
\usepackage{amsmath, amsthm, amssymb}
\usepackage{fontenc}
\usepackage[utf8]{inputenc}
\usepackage{graphicx}
\usepackage{xargs}
\usepackage[subrefformat=parens,labelformat=parens]{subcaption} 
\usepackage{algorithm, booktabs}
\usepackage[noend]{algpseudocode}
\usepackage{enumitem,listings,microtype,tikz,pgfplots}
\usepackage{hyperref}
\usepackage{environ,ifthen}
\usepackage{multirow} 
\usepackage{verbatim}
\usepackage[symbol]{footmisc}
\usepackage{commath}
\newtheorem{theorem}{Theorem}[section]

\newtheorem{lemma}[theorem]{Lemma}
\newtheorem{remark}[theorem]{Remark}

\newtheorem{corollary}[theorem]{Corollary}
\setkomafont{sectioning}{\rmfamily\bfseries}
\setkomafont{title}{\rmfamily}
\setkomafont{descriptionlabel}{\rmfamily\bfseries}

\newcommand{\N}{\mathbb{N}}

\newcommand{\RR}{\mathbb{R}}

\newcommand{\dist}{{\mathrm{d}}}
\def\Aplus{\tilde A}

\newcommand{\tT}{\mathrm{T}}
\DeclareMathOperator*{\argmin}{arg\,min}
\DeclareMathOperator*{\argmax}{arg\,max}

\DeclareMathOperator{\polar}{Polar}

\DeclareMathOperator{\dom}{dom}

\DeclareMathOperator{\tr}{tr}

\def\smallperp{\scriptscriptstyle\kern-1pt\bot\kern-1pt}

\begin{document}
\title{On the Rotational Invariant $L_1$-Norm PCA}

\author{Sebastian Neumayer\footnotemark[1]
\and
Max Nimmer\footnotemark[1] \footnotemark[4] 
\and 
Simon Setzer\footnotemark[3] 
\and 
Gabriele Steidl\footnotemark[1] \footnotemark[2]}

\maketitle
\footnotetext[1]{Department of Mathematics,
	Technische Universit\"at Kaiserslautern,
	Paul-Ehrlich-Str.~31, D-67663 Kaiserslautern, Germany,
	\{neumayer,nimmer,steidl\}@mathematik.uni-kl.de.
	} 
\footnotetext[2]{Fraunhofer ITWM, Fraunhofer-Platz 1,
	D-67663 Kaiserslautern, Germany}
\footnotetext[3]{Engineers Gate, London, United Kingdom}
\footnotetext[4]{corresponding author}

\begin{abstract}
    	Principal component analysis (PCA) is a powerful tool for dimensionality reduction.
		Unfortunately, it is sensitive to outliers, so that various robust PCA variants 
		were proposed in the literature.
		One of the most frequently applied methods for high dimensional data reduction is the rotational invariant $L_1$-norm PCA 
		of Ding and coworkers.
		So far no convergence proof for this algorithm was available.
		The main topic of this paper is to fill this gap.
		We reinterpret this robust approach as a conditional gradient algorithm
		and show moreover that it coincides with a gradient descent algorithm on
		Grassmannian manifolds.
		Based on the latter point of view, we prove  global
		convergence of the whole series of iterates 
		to a critical point using the Kurdyka-{\L}ojasiewicz property of the objective function,
		where we have to pay special attention to so-called anchor points, where the function is not differentiable.
				\vspace{5pt}
		
		\vspace{5pt}
	
\end{abstract}
%\linenumbers
%----------------------------------------------
\section{Introduction} \label{sec:intro}
%----------------------------------------------
In exploratory data analysis, principal component analysis (PCA) \cite{Pearson1901} still is one of the most popular tools for dimensionality reduction. 
Given $N \ge d $ data points $x_1,\ldots,x_N \in \mathbb R^d$, it finds a $K$-dimensional affine subspace
$\{\hat A \, t + \hat b: t \in \mathbb R^K\}$, $1 \le K \le d$, of $\mathbb R^d$
having smallest squared Euclidean distance from the data:
\begin{equation} \label{PCA_1}
	(\hat A, \hat b) \in \argmin_{A \in \mathbb R^{d,K}, b \in \mathbb R^d} \sum_{i=1}^N \min_{t \in \mathbb R^K} \|A\, t + b - x_i \|^2,
\end{equation}
where $\| \cdot\|$ denotes the Euclidean norm.
While $\hat b$ in the above minimization problem is not unique,
every minimizing affine subspace goes through 
the offset(bias) 
$\bar b \coloneqq \frac{1}{N}(x_1 + \ldots + x_N)$.
Therefore, we restrict our attention to data points 
$y_i \coloneqq x_i - \bar b$, $i=1,\ldots,N$, and subspaces through the origin minimizing the squared Euclidean distances to the $y_i$, $i=1,\ldots,N$.
Setting further the gradient with respect to $t \in \mathbb R^K$ to zero and allowing only orthonormal columns in $A$,
the PCA problem becomes
\begin{equation} \label{PCA_2}
	\hat A \in \argmin_{{A \in \mathbb R^{d,K}} \atop {A^\tT A = I_K}} \sum_{i=1}^N \|P^{\smallperp}_A y_i \|^2,
\end{equation}
where 
\[
	P^{\smallperp}_A \coloneqq I_d - A A^\tT
\]
denotes the orthogonal projection onto 
$\mathcal{R}(A)^\perp = \mathcal{N}(A^\tT)$ and $I_d$ the $d\times d$ identity matrix. 
Here $\mathcal{R}(A)$ denotes the range and $\mathcal{N}(A)$
the kernel of $A$.
One convenient property of PCA is the \emph{nestedness of the PCA subspaces}, i.e.,
for $K < \tilde K \le d$,
the optimal $K$-dimensional PCA subspace is contained in the $\tilde K$-dimensional one.
Furthermore, it is very fast, as the columns of an optimal $\hat A$ are the eigenvectors corresponding to the  $K$ 
largest eigenvalues of the empirical covariance matrix $\frac{1}{N-1}\sum_{i=1}^N y_iy_i^\tT$ which can be computed in polynomial time.

\begin{figure}[t]
	\begin{center}
		\includegraphics[width=0.45\textwidth]{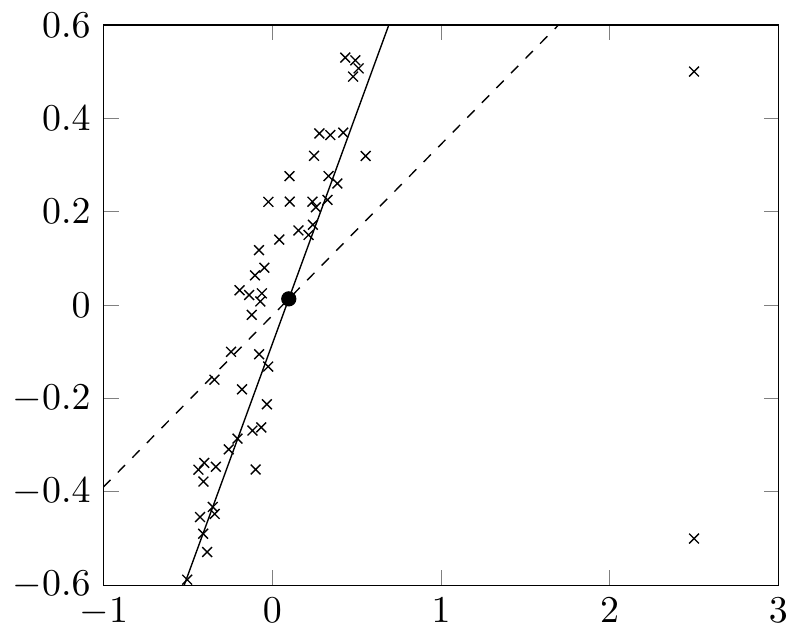}
	\end{center}
	\caption{Demonstration of the sensitivity of standard PCA to outliers. The data set consists of $50$ points close to a line through the origin and two outliers. 
	The subspace indicated by the dashed line  is the result of standard PCA \eqref{PCA_2}, while the solid one corresponds to \eqref{PCA_L1}. 
	In both cases the offset $b$ was chosen as the mean (black dot).}\label{fig:whyrobust}
\end{figure}

Unfortunately, PCA is sensitive to outliers appearing quite often in real-world data sets, see Fig. \ref{fig:whyrobust}.
A lot of different methods in robust statistics \cite{HR2009,RL1987,MMY2006} and optimization
were proposed to make the dimensionality reduction more robust.
One possibility consists of removing outliers before computing the principal components which has the serious drawback that outliers are difficult 
to identify and other data points are often falsely labeled as outliers.
Another approach  assigns different weights to data points based on their estimated relevancy,
to get a weighted PCA, see, e.g.~\cite{KKSZ08}. 
The RANSAC algorithm \cite{fischler1987random} repeatedly estimates the model parameters from a random subset of the data points until 
a satisfactory result is obtained as indicated by the number of data points within a certain error threshold. 
In a similar vein, least trimmed squares PCA models \cite{podosinnikova2014robust,rousseeuw2005robust} 
aim to exclude outliers from the squared error function, but in a deterministic way. 
The variational model in \cite{candes2011robust}
decomposes the data matrix $Y =(y_1 \ldots y_N)$ 
into a low rank and a sparse part.
Related approaches such as \cite{MT2011,XCS2012} separate the low rank component from the column sparse one using different norms in the variational model. However, such a decomposition is not always realistic.
From a statistical point of view, robust subspace recovery can be done via Tyler's M-estimator \cite{Tyler1987,Zhang2015}, 
a special case of M-estimators of covariance. But due to the large number of variables to be estimated for the scatter matrix, this approach is not feasible for high dimensional data.
Another group of robust PCA approaches replaces the squared $L_2$ norm in the PCA by the $L_1$ norm.
Then, the minimization of the energy functional can be addressed by linear programming, see, e.g.,  
Ke and Kanade \cite{ke2005robust}. Unfortunately, this norm is not rotationally invariant meaning that in general for an orthogonal matrix $Q$ we have $\|Qx\|_1 \neq \|x\|_1$. 
Consequently, when rotating the centered data points $y_i$, the minimizing subspace is not rotated in the same way.

\begin{figure}[t]
\begin{center}
	\includegraphics[width=0.5\textwidth]{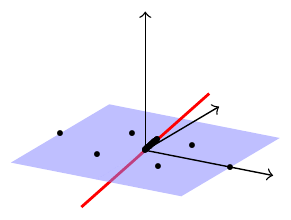}
\end{center}
	\caption{
		Global minimizing line/plane corresponding to $K=1,2$ in \eqref{PCA_L1}
		demonstrating the loss of the nestedness property of the PCA. 
		The data $y_i \in \mathbb R^3$ (black dots) 
		are given by the points $\{(0.005l,0,0.005l): l\in\{0,\ldots, 30\}\}$
		on a line along with  points $(1,0,0)^\tT$, $(-1,0,0)^\tT$, $(1/\sqrt{2}, 1/\sqrt{2},0)^\tT$, 
		$(-1/\sqrt{2},-1/\sqrt{2},0)^\tT$, $(1/\sqrt{2},-1/\sqrt{2},0)^\tT$, $(-1/\sqrt{2},1/\sqrt{2},0)^\tT$ on an ellipse. 
		The direction $\hat a_1 = (1/\sqrt{2},0,1/\sqrt{2})^\tT$ 
		(red line) does not lie on the blue plane generated by 
		$\hat A$ which has the columns $(1, 0,0 )^\tT$ and $(0,1,0)^\tT$.
		\label{fig:robPCA_model_diff}
	}
\end{figure}

In this paper, the focus lies on the model
\begin{equation} \label{PCA_L1_offset}
	(\hat A, \hat b) \in \argmin_{A \in \mathbb R^{d,K}, b \in \mathbb R^d} \sum_{i=1}^N \min_{t \in \mathbb R^K} \|A\, t + b - x_i \|
\end{equation}
where in contrast to \eqref{PCA_1} we do not square the Euclidean norm.
First of all, let us mention that the determination of the offset $b \in \mathbb R^d$ is not straightforward now.
Frequently, the geometric median of the data is used as offset, which is in general not a minimizer of \eqref{PCA_L1_offset}, see \cite[Sect.~5]{NNSS19b}.
However, in this paper it is assumed that $b \in \mathbb R^d$ is fixed and already subtracted from the data. 
Then, \eqref{PCA_L1_offset} reduces to
\begin{equation} \label{PCA_L1}
	\hat A \in \argmin_{{A \in \mathbb R^{d,K}} \atop {A^\tT A = I_K}} \sum_{i=1}^N  \|P^{\smallperp}_A y_i \|.
\end{equation}
A slightly different form of this model became popular under the name \emph{rotational invariant $L_1$-norm PCA} by a paper of Ding et al.~\cite{ding2006r}.

It is important to note that in contrast to the classical PCA the hierarchical structure of the approach is lost. 
This is exemplified in Fig.~\ref{fig:robPCA_model_diff}.
We mention that several models applying the deflation technique of standard PCA in a robust setting
were also provided in the literature. These models are not of interest in this paper,
but we refer to the book \cite[p.~203]{Huber1981} 
and the collection of papers \cite{HFB2014,KK2016,kwak2008principal,LCTZ15,LS1985,MT2011,nie2011robust}
which is clearly not complete.

Despite the loss of the nestedness property and the non-smooth target function which is harder to optimize, 
the robustness makes \eqref{PCA_L1} 
an attractive choice in practical problems. 
Examples of improved performance compared 
to standard PCA and other robust alternatives can be found in, e.g., \cite{ding2006r,LM2018}. Furthermore, in \cite{MZL19} it was shown that 
under certain assumptions on the given so-called inlier and outlier data, by minimizing \eqref{PCA_L1} 
we can recover the exact underlying subspace spanned by the inlier data points. 
Additionally, in contrast to convex relaxation methods, the approach is well suited for high dimensional data appearing, e.g., in image processing.

The minimization of \eqref{PCA_L1} has been treated before. Ding et al.~\cite{ding2006r} suggest a constrained minimization based on a power method without convergence proof. 
In \cite{MZL19} the optimization problem is tackled by a geodesic gradient descent approach on the Grassmannian and under certain assumptions on the data, 
local convergence to the global minimizer is shown for an appropriate starting step size and initial iterate.
A tight convex relaxation was studied in \cite{LCTZ15}, 
where the projection matrix $AA^\tT$ is estimated. Due to the much larger size of the projection matrix, this approach is not suitable for high-dimensional data. 
An approach based on iteratively reweighted least squares can be found in \cite{LM2017}, where standard PCA is repeatedly applied to rescaled data points. For the special case of one-dimensional subspaces ($K=1$), a Weiszfeld-like algorithm with convergence was given in \cite{NNSS19b}.

In this paper, the iterative algorithm of Ding et al.~\cite{ding2006r} is first interpreted as a conditional gradient algorithm, also known as Frank-Wolfe algorithm, 
which only implies a certain convergence behavior of subsequences of the iterates.
Recalling that we are not interested in the columns of the minimizer $\hat A$ in \eqref{PCA_L1} itself, but just
in the $K$-dimensional subspace spanned by them,
we show that the algorithm  can be recast as a gradient descent algorithm on the Grassmannian manifold. This enables us to prove global convergence of the whole sequence of iterates under mild assumptions. % provided that no point is reached where the functional in \eqref{PCA_L1} is not differentiable.

The paper is organized as follows: In Section \ref{sec:prelim},
we recall preliminaries on Stiefel and Grassmannian manifolds.
We discuss important properties of the robust PCA functional in Section \ref{sec:model}.
Section \ref{sec:mini_1} shows the equivalence of
the algorithm of Ding et al.~\cite{ding2006r},
the conditional gradient algorithm and a gradient descent algorithm on Grassmannians.
The proof of global convergence of the whole sequence of iterates under some restrictions on the so-called anchor points is given in Section \ref{sec:convergence}.
Finally, Section \ref{sec:ana_anchor} addresses the topic of anchor points.

%----------------------------------------------
\section{Preliminaries on Stiefel and Grassmannian Manifolds} \label{sec:prelim}
%----------------------------------------------
In this section, we briefly provide the basic notation on Stiefel and Grassmannian manifolds which is
required in our approach. Good references on the topic, 
in particular for optimization on these manifolds, are
\cite{AMS08,EAS98}.

Let $K \le d$. The (compact) \emph{Stiefel manifold} is defined by
\[
	S_{d,K} \coloneqq \left\{A \in \mathbb R^{d,K}: A^\tT A = I_K\right\}.
\]
For $K=1$, it coincides with the unit sphere $S_{d,1} = S^{d-1}$ in $\mathbb R^d$
and for $K=d$ with the orthogonal matrices $S_{d,d} = O(d)$.
The tangent space at $A \in S_{d,K}$ is given by
\begin{align}
	T_A S_{d,K}
	&\coloneqq \left\{H \in \mathbb R^{d,K}: H^\tT A + A^\tT V = 0\right\}\\
	&= \left\{H \in \mathbb R^{d,K}: H = A X + A_\perp Z, \quad X \in \mathbb R^{K,K} \; \mbox{skew symmetric} ,\, Z \in \mathbb R^{d-K,K}\right\},
\end{align}
where $A_\perp$ denotes a matrix with orthonormal columns which are in addition orthogonal to the columns of $A$.
There are two common ways to define inner products on the tangent space such that $S_{d,K}$
becomes a Riemannian manifold, namely
\begin{itemize}
 \item[i)]  the Frobenius inner product
$\langle H_1,H_2\rangle_F \coloneqq \tr(H_1^\tT H_2)$,
or
\item[ii)]
the canonical inner product
$
\langle H_1,H_2\rangle_A \coloneqq \tr\bigl(H_1^\tT (I_d - \frac12 A A^\tT) H_2\bigr) 
$.
\end{itemize}
In the rest of this paper $\|A\| = \tr(A^\tT A)^{\frac12}$ always denotes the Frobenius norm of $A$.
The first inner product appears when considering $S_{d,K}$ as submanifold of the Euclidean space $\mathbb R^{d,K}$,
while the second one relies the quotient structure $S_{d,K}= O(d)/O(d-K)$.
We are mainly interested in the $K$-dimensional subspace spanned by the columns of $A \in S_{d,K}$,
which does not change if we multiply $A$ from the right with an orthogonal matrix $Q \in O(K)$.
This is pictured by the \emph{Grassmannian manifold}, or just Grassmannian, which can be defined as
quotient manifold of the Stiefel manifold $G_{d,K} \coloneqq S_{d,K}/O(K)$. The equivalence classes
$[A]\coloneqq \{AQ :\ Q\in O(K)\}$ 
belonging to $G_{d,K}$
can be represented by elements $A$ of the Stiefel manifold.
The tangent space of $G_{d,K}$ at $[A]$ can be identified with its horizontal lift at $A$,
\[
	T_{[A]} G_{d,K} \coloneqq \left\{A_\perp Z: Z\in\mathbb{R}^{d-K,K} \right\}.
\]
Further, the Grassmannian becomes a Riemannian manifold by reducing the Riemannian metrics in i) or 
equivalently ii) to $T_{[A]} G_{d,K}$, i.e., for any representative $A \in S_{d,K}$ and $H_1,H_2 \in T_{[A]} G_{d,K}$, 
\[
	\langle H_1,H_2 \rangle_{[A]} \coloneqq \langle H_1,H_2 \rangle_A = \tr\Bigl(H_1^\tT \bigl(I_K - \frac 12 AA^\tT\bigr) H_2\Bigr) 
	= \tr\bigl(H_1^\tT H_2\bigr) = \langle H_1,H_2 \rangle_F.
\]
A possible choice for a metric on the Grassmannian is given by
\[
	d_{G_{d,K}}\bigl([A_1],[A_2]\bigr) \coloneqq \bigl\|A_1A_1^\tT - A_2A_2^\tT\bigr\|_2,
\]
where $A_1,A_2\in S_{d,K}$ and $\|\cdot\|_2$ is the spectral norm.\\
In PCA we aim to find an optimal subspace, which means that we are interested in elements of Grassmannians. 
In practice, working with equivalence classes is difficult and hence calculations are performed with representatives on the Stiefel manifold. 

The proposed optimization algorithms involve the orthogonal projection 
$\Pi_{S_{d,K}}\colon \mathbb R^{d,K} \rightarrow S_{d,K}$, i.e.,
\[
 	\Pi_{S_{d,K}}(M) = \argmin_{O\in S_{d,K}} \|M-O\|^2 =\argmax_{O\in S_{d,K}} \langle O,M \rangle.
\]
To this end, recall that the polar decomposition of
a matrix $M\in\mathbb{R}^{d,K}$ is given by $M = QS$, 
where $Q\in S_{d,K}$ and $S\in\mathbb{R}^{K,K}$ is symmetric and positive semi-definite. 
Starting with the (economy-size) singular value decomposition $M = U\Sigma V^\tT$, 
where $U\in S_{d,K}$, $\Sigma \in \RR^{K,K}$ is a diagonal matrix
and	and $V\in S_{K,K}$, 
the polar decomposition is determined by $Q\coloneqq \polar(M) \coloneqq UV^\tT$ and $S \coloneqq V \Sigma V^\tT$. 
The following lemma can be found e.g.~in \cite{JNRS10,podosinnikova2014robust}.

\begin{lemma}\label{lem:orth_proj}
	The orthogonal projection $\Pi_{S_{d,K}}\colon \mathbb R^{d,K} \rightarrow S_{d,K}$ is given by
	\[
	\Pi_{S_{d,K}}(M) = \polar(M).
	\]
	If $M$ has full rank, then $\polar(M) = M(M^\tT M)^{-\frac 12}$.
\end{lemma}

%----------------------------------------------
\section{Model Analysis} \label{sec:model}
%----------------------------------------------
The main focus of this section lies on investigating the objective function
in \eqref{PCA_L1} and a related function
with respect to differentiability and convexity.
To be precise, we are actually only interested in minimizing over equivalence classes $[A]\coloneqq 	\{AQ :\ Q\in O(K)\}$.
Besides the objective function
\begin{equation} \label{E}
	E(A) = \sum_{i=1}^N E_i(A) \coloneqq  \sum_{i=1}^N \| P^{\smallperp}_A y_i \| = \sum_{i=1}^N \bigl\| (I_d - A A^\tT) y_i \bigr\| ,
\end{equation}
we also deal with the function
\begin{equation} \label{F}
	F(A) = \sum_{i=1}^N F_i(A) \coloneqq \sum_{i=1}^N \sqrt{y_i^\tT P^{\smallperp}_A y_i} = \sum_{i=1}^N \sqrt{\| y_i \|^2- \| A^\tT y_i\|^2}.
\end{equation}
Clearly, these two functions take the same values on the Stiefel manifold $S_{d,K}$.
However, they have quite different properties as functions on $\mathbb R^{d,K}$, but this was often neglected in existing approaches.
While $E$ is well-defined on the whole $\mathbb R^{d,K}$, the function 
$F$ is only well defined on the closed domain
\begin{equation}
	{\mathcal D} \coloneqq \bigcap\limits_{i=1}^N {\cal D}_i, \quad {\mathcal D}_i \coloneqq \left\{A \in \mathbb R^{d,K}: \| y_i \|^2- \| A^\tT y_i\|^2 \ge 0\right\}
\end{equation}
and therefore it is extended to the whole $\mathbb R^{d,K}$ by
\begin{equation} \label{ext_F}
	F(A) \coloneqq -\infty \quad \mathrm{if} \quad A \not \in  {\mathcal D}.
\end{equation}
For $A \in S_{d,K}$ and all $y \in \mathbb R^d$, it holds $\|A^\tT y\| \le \|y\|$ so that 
$S_{d,K} \subset {\mathcal D}$.
Further, $A \in S_{d,K} \cap \partial{\mathcal D}$ if and only if 
$\|P^{\smallperp}_A y_i\| = 0$ for some $i \in \{1,\ldots,N\}$.
The compact subset of $\mathbb R^{d,K}$
\begin{equation} \label{anchor}
	\mathcal A \coloneqq \bigl\{A \in S_{d,K}: \|P^{\smallperp}_A y_i\| = 0 \; \mbox{for some} \; i \in \{1,\ldots,N\} \bigr\}
\end{equation}
is called \emph{anchor set}.

In the simple case $N=d=K=1$ and $y_1=1$, the above functions read $E(A) = |1-A^2|$ and $F(A) = \sqrt{1-A^2}$ with $A \in \mathbb R$.
While the first function is locally Lipschitz continuous on $[-1,1]$, the second one does not have this property at $A = \pm 1$. 
The following two lemmata state properties of $E$ and $F$.

\begin{lemma} \label{propE}
The function $E$ defined by \eqref{E} is
locally Lipschitz continuous on $\mathbb R^{d,K}$.	
\end{lemma}

\begin{proof}
It suffices to show the property for the summands $E_i$. For an arbitrary fixed $A \in S_{d,K}$,
let $\|A- A_i\| \le \varepsilon$, $i=1,2$.
Then, we obtain
\begin{align*}
	|E_i(A_1) - E_i(A_2)| &= \bigl| \, \|P^{\smallperp}_{A_1} y_i\| - \|P^{\smallperp}_{A_2} y_i\| \, \bigr|
	\le \|P^{\smallperp}_{A_1} y_i - P^{\smallperp}_{A_2} y_i\| \\
	&\le \bigl\|A_1 A_1^\tT - A_2 A_2^\tT\bigr\| \|y_i\| \\
	&= \frac12 \bigl\|(A_1 - A_2)(A_1^\tT + A_2^\tT)  + (A_1+ A_2)(A_1^\tT - A_2^\tT) \bigr\| \| y_i\|\\
	&\le 2 (\|A\| + \varepsilon) \|y_i\| \|A_1 - A_2\|. 
\end{align*}
\end{proof}

In general, the function $E$ is neither convex nor concave on ${\mathcal D}$, see Fig.~\ref{fig:con}.
\begin{figure}
\centering
	\includegraphics[width=0.5\textwidth]{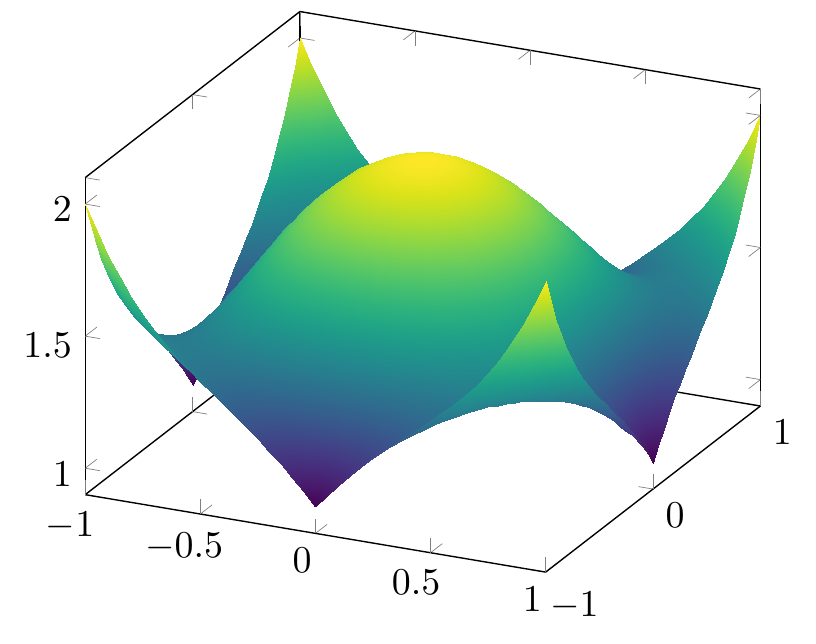}
	\caption{Plot of $E(A)$ on $\mathcal{D}$ for $d=2$ with data points $y_1=(1,0)^\tT$ and $y_2=(0,1)^\tT$. \label{fig:con} }
\end{figure}
In contrast, $F$ is concave as the following lemma shows.

\begin{lemma} \label{propF}
The function $F$  defined by \eqref{F} and \eqref{ext_F} fulfills the following relations:
	\begin{enumerate}
		\item $\mathrm{dom}  (-F) \coloneqq \{A \in \mathbb R^{d,K}: -F(A) < + \infty\} = {\mathcal D}$  is convex.
		\item $-F$ is convex.
		\item The subdifferential of $-F$ is empty at the boundary of ${\mathcal D}$, i.e.~at $A \in \mathbb R^{d,K}$ with $\|y_i\|^2 - \|A^\tT y_i\|^2 = 0$ for some $i \in \{1,\ldots,N\}$.
	\end{enumerate}
\end{lemma}

\begin{proof}
\begin{enumerate}
			\item 
			It holds $A \in \mathrm{dom} (-F)$ if and only if 
			\begin{equation} \label{est}
				 \|y_i\|^2 - \|A^\tT y_i\|^2 \geq 0
			\end{equation}
			for all $i=1,\ldots,N$. 
			Since the intersection of convex sets is convex,
			it suffices to show convexity of $\mathrm{dom} (- F_i)$ separately.
			Let $A_1,A_2\in \mathrm{dom} (- F_i)$. 
			Then, using \eqref{est}, we obtain for $\lambda \in [0,1]$ that
			\begin{align*}
				 &\|y_i\|^2 - \|(\lambda A_1 + (1-\lambda) A_2)^\tT y_i\|^2\\
				&= 
				\|y_i\|^2 - \left( \lambda^2 \|A_1^\tT y_i\|^2 + 2\lambda(1-\lambda)\langle A_1^\tT y_i, A_2^\tT y_i \rangle + (1-\lambda)^2\|A_2^\tT y_i\|^2 \right)\\
				&\geq 
				\|y_i\|^2 - \left( \lambda^2 \|A_1^\tT y_i\|^2 + 2\lambda(1-\lambda)\| A_1^\tT y_i\| \|A_2^\tT y_i\| + (1-\lambda)^2\|A_2^\tT y_i\|^2 \right)\\
				&\geq 
				\|y_i\|^2 - \left( \lambda^2 \|y_i\|^2 + 2\lambda(1-\lambda)\|y_i\|^2 + (1-\lambda)^2\|y_i\|^2 \right) = 0.
			\end{align*}
			Thus, $\lambda A_1 + (1-\lambda) A_2 \in \mathrm{dom} (-F_i)$ and the claim follows.
			
			\item 
			Since the sum of concave functions is concave again,  it suffices to consider the individual summands $F_i$ again.
			For $\varepsilon > 0$, we define 
			$$F_\varepsilon(A) \coloneqq \sqrt{\|y_i\|^2 - \|A^\tT y_i\|^2 + \varepsilon},$$ 
			which is differentiable on an open set containing ${\mathcal D}_i$. 
			By the chain rule and since 
			$\frac{\partial}{\partial A}\tr(y_i^\tT AA^\tT y_i) = 2y_i y_i^\tT A$, 
			the gradient of $F_\varepsilon$ is
			\[
			\nabla F_\varepsilon(A) = -\frac{1}{F_\varepsilon(A)}y_iy_i^\tT A.
			\]
			Using the product rule and the chain rule, the Hessian is given by
			\[
			\nabla^2 F_\varepsilon(A)[H] = -\frac{1}{F_\varepsilon(A)^2}\Bigl( y_iy_i^\tT H F_\varepsilon(A) + \frac{1}{F_\varepsilon(A)}\langle y_iy_i^\tT A,H \rangle y_iy_i^\tT A \Bigr),
			\]
			for all $H \in \mathbb R^{d,K}$ so that
			\begin{align*}
			\bigl\langle \nabla_A^2 F_\varepsilon(A)[H],H \bigr\rangle 
			&= -\frac{1}{F_\varepsilon(A)^2}\Bigl( F_\varepsilon(A)\langle y_iy_i^\tT H,H\rangle  + \frac{1}{F_\varepsilon(A)}\langle y_iy_i^\tT A,H \rangle^2 \Bigr)\\
			&= -\frac{1}{F_\varepsilon(A)^2}\Bigl( F_\varepsilon(A)\|H^\tT y_i\|^2  + \frac{1}{F_\varepsilon(A)}\langle y_iy_i^\tT A,H \rangle^2 \Bigr)
			\leq 0.
			\end{align*}
			Consequently, the Hessian is negative semidefinite and $F_\varepsilon$ concave in ${\mathcal D}_i$ for all $\varepsilon>0$. Finally,
			\[
				F_i = \inf \{ F_\varepsilon: \varepsilon>0 \}
			\]
			is concave as the pointwise infimum of a family of concave functions.
			
			\item For an arbitrary fixed $i \in \{1,\ldots,N\}$, let $A_0 \in \mathbb R^{d,K}$ with $\| y_i \|^2 - \| A_0^\tT y_i\|^2 = 0$.
			We consider the subdifferential of $F_i$ at $A_0$ given by
			\begin{align*}
			\partial F_i(A_0) 
			&= \left\{P \in \mathbb R^{d,K}: - F_i(A) \ge -F_i(A_0) + \langle P,A-A_0 \rangle \quad \forall \; A \in \mathbb R^{d,K} \right\}
			\\
			&=  \Bigl\{P \in \mathbb R^{d,K}: -\sqrt{\| A_0^\tT y_i\|^2 - \| A^\tT y_i\|^2} \ge \langle P,A-A_0 \rangle 
			\quad \forall \; A \in \mathbb R^{d,K} \Bigr\}.
			\end{align*}
			Choosing $A \coloneqq \alpha A_0$ with $\alpha \in [0,1]$, a subgradient $P$ must fulfill
			\begin{align*}
			 -\| A^\tT y_i\| \sqrt{1-\alpha^2} &\ge  (\alpha - 1) \langle P,A_0 \rangle,\\
			 \| A^\tT y_i\| \sqrt{1+\alpha} &\le \sqrt{1-\alpha}  \; \langle P,A_0 \rangle,
			 \end{align*}
			 which leads to a contradiction if $\alpha \rightarrow 1$. Hence, the subdifferential is empty.
	\end{enumerate}
\end{proof}

For the algorithms the gradient and the Riemannian gradient on the Grassmannian of the functions $E$ and $F$ are required.

\begin{lemma} \label{lem:derivatives_E_F}
Let $E$ and $F$ be defined by \eqref{E} and \eqref{F}, respectively.
Then, the gradient $\nabla$
and the Riemannian gradient $\nabla_A$ on $S_{d,K}$ 
at $A \in S_{d,K}\backslash {\mathcal A}$ are given by
\begin{equation} \label{der_E_F}
\nabla F(A) = - C_A \, A, \quad
\nabla E(A) = \nabla_{A} E (A) = \nabla_{A} F(A) = - P^{\smallperp}_A \, C_A \, A ,
\end{equation}
where
\begin{equation} \label{eq:def_C}
C_{A} \coloneqq \sum_{i=1}^N \frac{y_i y_i^\tT}{\|P^{\smallperp}_A y_i\|}.
\end{equation}
\end{lemma}

Note that $- P^{\smallperp}_A \, C_A \, A$ is also the horizontal lift of the gradient
$\nabla_{[A]} \tilde E ([A])$ on the Grassmannian at $A$, where $E = \tilde E \circ \pi$
and $\pi$ is the projection from $S_{K,d}$ onto $G_{K,d}$.

\begin{proof}
By straightforward computation we obtain for $A \in \mathbb R^{d,K}$ that
 \begin{align*}
\nabla E_i(A) &= - \frac{1}{\|P^{\smallperp}_A y_i\|} \left( P^{\smallperp}_A y_i y_i^\tT A + y_i y_i^\tT P^{\smallperp}_A A \right) \quad \mbox{if} \quad P^{\smallperp}_A y_i \not = 0,\\
\nabla F_i(A) &= - \frac{1}{ \bigl( \|y_i\|^2 - \|A^\tT y_i\|^2 \bigr)^{-\frac12}}  y_i y_i^\tT A  \qquad \; \; \mbox{if} 
\quad A \in  \mathrm{int}\bigl(\mathrm{dom} (-F_i)\bigr) = \mathrm{int} (\mathcal{D}_i).
\end{align*}
For $A \in S_{d,K}$ the gradient of $E_i$ coincides with the Riemannian gradient of $E_i$ on $S_{d,K}$, i.e.,
$$
\nabla_{A} E_i(A) = - \frac{1}{\|P^{\smallperp}_A y_i\|}  P^{\smallperp}_A y_i y_i^\tT A \quad \mbox{for} \quad A \in S_{d,K}.
$$
This implies the assertion.
\end{proof}

We call $A \in  S_{d,K} \backslash \mathcal{A}$ a critical point of $F$, resp.~$E$
if 
\begin{equation} \label{crit_point}
\nabla_{A} E (A) = \nabla_{A} F(A) = - P^{\smallperp}_A \, C_A \, A = 0.
\end{equation}

%----------------------------------------------
\section{Minimization Algorithm} \label{sec:mini_1}
%----------------------------------------------
In this section, we show that the constrained minimization algorithm of Ding et al.~\cite{ding2006r} can be interpreted as
i) a conditional gradient algorithm, and ii) a gradient descent algorithm on Grassmannians.
The conditional gradient algorithm, also known as Frank-Wolfe algorithm, was originally proposed 1956 in \cite{FW1956} for solving linearly constrained quadratic programs 
and was later adapted to other problems.
For a good overview, we refer to \cite{LT13} and the references therein.
Then it follows from general results on the algorithm that \emph{a subsequence of the iterates} converges under strong conditions on the anchor points
to a critical point of the functional.

\paragraph{Constrained Minimization Algorithm.}
Ding et al.~\cite{ding2006r} consider the constrained optimization problem
\begin{equation}
\argmin_{A \in \mathbb R^{d,K}} F(A) \quad \mbox{subject to} \quad A^\tT A = I_K.
\label{eq:model_constr}
\end{equation}
The authors claimed 
%before and in Proposition 0 
without proof 
that the function $F$ is convex in $A A^\tT$
and has a unique global minimizer. 
Both statements are not correct: for $N=K=d=1$ and $y_1 = 1$ it is easy to check that $F(A) = \sqrt{1-A^2}$ is concave in $A^2 \in [0,1]$;
for $N=2$, $K=1$, $d=2$ with centered data points $y_1 = (-1/2,\sqrt{3}/2)^\tT$ and $y_2 = (1/2,\sqrt{3}/2)^\tT$ the
minimizers of $F(A)$ are given by $A = (-1/2,\sqrt{3}/2)^\tT$ and $A= (1/2,\sqrt{3}/2)^\tT$, which span different subspaces.
Penalizing the constraint in \eqref{eq:model_constr} via a symmetric Lagrange multiplier $\Lambda \in \mathbb R^{K,K}$,
setting the gradient of the resulting Lagrangian
$L(A,\Lambda) \coloneqq F(A) + \langle \Lambda, A^T A - I_K\rangle$ with respect to $A$ to zero and applying an orthogonalization procedure, see Lemma~\ref{lem:orth_proj}, the authors 
arrive at the following iteration scheme: if $A^{(r)} \not \in \mathcal{A}$,
\begin{equation} \label{eq:scheme_ding}
 A^{(r+1)} \coloneqq \Pi_{S_{d,K}}\left(C_{A^{(r)}}  A^{(r)}\right) =  C_{A^{(r)}} A^{(r)} \Big( \big(A^{(r)} \big)^\tT C_{A^{(r)}}^2 A^{(r)} \Big)^{-\frac12}.
\end{equation}

\paragraph{Conditional Gradient Algorithm.}
The conditional gradient algorithm is commonly used to minimize a convex function over a compact set. 
However, as in \cite{LT13}, we apply it for maximizing the convex function $-F$.

In general, for a proper convex function $f\colon \mathbb R^n \rightarrow \mathbb R \cup \{+\infty\}$ and 
a nonempty, compact set 
${\mathcal{C}} \subset \mathrm{int}(\dom f)$, the
conditional gradient algorithm is the update scheme
\begin{equation}\label{eq:condgrad}
	u^{(r+1)} \in \argmax_{u \in {\mathcal{C}}} \left\langle  u- u^{(r)}, p^{(r)} \right\rangle = \argmax_{u \in {\mathcal{C}}} \left\langle  u, p^{(r)} \right\rangle , \quad p^{(r)} \in \partial f\bigl(u^{(r)}\bigr).
\end{equation}
Note that according to \cite[Corollary~32.4.1]{Ro1970}, the value $\hat u \in {\mathcal{C}}$ is a local maximizer
of $f$ over ${\mathcal{C}}$ if for all $v \in {\mathcal{C}}$,
\begin{equation}\label{cond_min}
 \bigl\langle v - \hat u, \hat p\bigr\rangle \le 0 \quad \mbox{for all} \quad \hat p \in \partial f (\hat u).
\end{equation}
By definition of the subdifferential we have
\begin{equation} \label{cond_cond}
 f\bigl(u^{(r+1)}\bigr) - f\bigl( u^{(r)}\bigr) \ge \left\langle u^{(r+1)} - u^{(r)},p^{(r)} \right\rangle = \max_{v \in {\mathcal{C}}} \left\langle v-u^{(r)},p^{(r)} \right\rangle \ge 0,
\end{equation}
where the last equation follows by choosing $v=u^{(r)} \in {\mathcal{C}}$.

For finite convex functions $f\colon \mathbb R^n \rightarrow \mathbb R$ 
the following convergence result was proved in \cite{LT13} based on \cite{Mang1996}.
The proof can be modified for $f$ with values in the extended real numbers and ${\mathcal{C}}\subset \mathrm{int}(\dom f)$
in a straightforward way.

\begin{theorem}\label{thm:conv_cond_g}
Let $f\colon \mathbb R^n \rightarrow \mathbb R \cup \{+\infty\}$
a proper convex function and ${\mathcal{C}} \subset \mathrm{int}(\dom f)$
a nonempty, compact set.
Then the sequence $\{f(u^{(r)})\}_r$ generated by \eqref{eq:condgrad} is strictly increasing
except when $\max_{u \in {\mathcal{C}}}  \langle u- u^{(r)}, p^{(r)}\rangle = 0$,  
in which case it terminates at $u^{(r)}$ satisfying \eqref{cond_min}.
If $f$ is continuously differentiable on $\mathrm{int}(\dom f)$,
then every accumulation point $\hat u$ of the sequence $\{u^{(r)} \}_r$ fulfills \eqref{cond_min}.
\end{theorem}

We  want to apply the scheme \eqref{eq:condgrad} for $f \coloneqq -F$ and
${\mathcal{C}} = {\mathcal{C}}_\varepsilon \coloneqq S_{d,K} \backslash {\mathcal A}_\varepsilon \subset \mathrm{int}\left(\dom (-F) \right)$,
where 
$${\mathcal A}_\varepsilon\coloneqq \bigl\{ A \in \mathbb S_{d,K}: \mathrm{dist}(A,{\mathcal A}) < \varepsilon\bigr\}$$
denotes the set of matrices in $S_{d,K}$
having a distance smaller than some $\varepsilon >0$ from the anchor set.
To this end, the maximization problem \eqref{eq:condgrad} has to be solved, but first we have to find a suitable $\varepsilon$. For this purpose define the iteration
\begin{equation}\label{eq:condgrad_E}
	A^{(r+1)} = \argmax_{U \in S_{d,K}} \left\langle U, \nabla (-F) \bigl( A^{(r)}\bigr) \right\rangle = \argmax_{U \in S_{d,K}} \left\langle U,C_{A^{(r)}} A^{(r)} \right\rangle
\end{equation}
for $A^{(r)}\notin \mathcal{A}$, where we plugged $\partial F(A^{(r)}) = \{\nabla F(A^{(r)})\}$ as given in Lemma~\ref{lem:derivatives_E_F} into \eqref{eq:condgrad}.
Now, assume that we can find $\varepsilon>0$ such that $A^{(r)}\in \mathcal{C}_\varepsilon$ for all $r\geq 0$. In this case
\[
	A^{(r+1)} = \argmax_{U \in \mathcal{C}_\varepsilon} \left\langle U,C_{A^{(r)}} A^{(r)} \right\rangle.
\]
Note that we can always find such an $\varepsilon$ for $r$ large enough if all accumulation points of $A^{(r)}$ as in \eqref{eq:condgrad_E} are non-anchor points.
Using Lemma \ref{lem:orth_proj} we obtain
\begin{align}\label{cond_grad=ding}
A^{(r+1)}
	= \Pi_{S_{d,K}}\bigl(C_{A^{(r)}}  A^{(r)}\bigr) 
	=  C_{A^{(r)}} A^{(r)} \Big( \big(A^{(r)} \big)^\tT C_{A^{(r)}}^2 A^{(r)} \Big)^{-\frac12},
\end{align}
which is exactly the iteration scheme \eqref{eq:scheme_ding} proposed by Ding et al.
Based on Theorem \ref{thm:conv_cond_g}, we have the following corollary for our special setting.

\begin{corollary} \label{cor:cond_grad}
 Let $F$ be defined by \eqref{F}-\eqref{ext_F}. Assume that the sequence $\{A^{(r)} \}_r$ generated
 by \eqref{cond_grad=ding} has no element in ${\mathcal A}$ and that the set of accumulation points has a positive distance from ${\mathcal A}$.
 Then the sequence $\{F \left( A^{(r)} \right) \}_r$ is strictly decreasing except for iterates where
 \begin{equation} \label{bed_1}
  \biggl \langle C_{A^{(r)}} A^{(r)} \left( \big(A^{(r)} \big)^\tT C_{A^{(r)}}^2 A^{(r)} \right)^{-\frac12} - A^{(r)}, C_{A^{(r)}} A^{(r)} \biggr\rangle =0,
 \end{equation}
in which case the iteration terminates at $A^{(r)}$ which is a critical point.
Condition \eqref{bed_1} is equivalent to $A^{(r+1)} = A^{(r)}$, resp.~to $\nabla_{A^{(r)}} F (A^{(r)}) = 0$.
If the iteration does not terminate after a finite number of steps, every accumulation point of $\{A^{(r)} \}_r$ is a critical point.
\end{corollary}

\begin{proof}
We show that the three stopping criteria are equivalent. 
Let $ A\coloneqq A^{(r)} \not \in \mathcal{A}$ and
recall that $\nabla_{A} F (A) = -P^{\smallperp}_{A} C_{A} A$.
\\
1. If $A^{(r+1)} =  A$, then $C_A A = A (A^\tT C_A^2 A)^\frac12$ and thus $P^{\smallperp}_A C_A A  =  0$.
\\
2. If $P^{\smallperp}_A C_A A  =  0$, 
then
$C_A A = A (A^\tT C_A A)$ 
and thus
$$A^{(r+1)} = C_A A \bigl(A^\tT C_A^2 A\bigr)^{-\frac12} = A (A^\tT C_A A) \left( A^\tT C_A  A (A^\tT C_A A)\right)^{-\frac12} = A.$$
Further, this implies \eqref{bed_1}.
\\
3. Assume now that \eqref{bed_1} is fulfilled. Then, we have
\begin{align}
0&= \Bigl \langle C_{A} A \bigl(  A ^\tT C_{A}^2 A \bigr)^{-\frac12} - A, C_{A} A \Bigr\rangle 
= \Bigl \langle \bigl(A^\tT C_A^2 A\bigr)^{\frac12} - A^\tT C_A A, I_K \Bigr\rangle\\
&= \tr \Bigl( \bigl(A^\tT C_A^2 A\bigr)^{\frac12} - A^\tT C_A A \Bigr).
\end{align}
On the other hand, we have with $\tr(AB) = \tr(BA)$ that
\begin{align}
 \|P^{\smallperp}_A C_A A \|^2 
 &= \left\langle P^{\smallperp}_A C_A A,P^{\smallperp}_A C_A A\right\rangle 
 = \tr \left( A^\tT C_A^2 A - \bigl(A^\tT C_A A\bigr)^2 \right)\\
 &= \tr\left( \Bigl(\bigl(A^\tT C_A^2 A\bigr)^{\frac12} - A^\tT C_A A \Bigr) \Bigl(\bigl(A^\tT C_A^2 A\bigr)^{\frac12} + A^\tT C_A A \Bigr) \right) \\
 &\le \lambda_{\max} \, \tr \Bigl(\bigl(A^\tT C_A^2 A\bigr)^{\frac12} - A^\tT C_A A \Bigr) = 0,
\end{align}
where $\lambda_{\max}$ denotes the largest eigenvalue of the matrix $(A^\tT C_A^2 A)^{\frac12} + A^\tT C_A A$.
Hence, $P^{\smallperp}_A C_A A = 0$.\\
It remains to show that all accumulation points of infinite sequences are critical points. 
Let $\hat A\notin \mathcal{A}$ be an accumulation point of such a sequence. 
As the set of accumulation points has positive distance from $\mathcal{A}$, we can choose $\varepsilon$ small enough such that all iterates 
are in $\mathrm{int}(\mathcal{C}_\varepsilon)$ for $r$ large enough. 
Then, by Theorem \ref{thm:conv_cond_g}, the accumulation point $\hat A$ fulfills \eqref{cond_min} 
and as $E$ is differentiable in $\hat A$, this implies that $\hat A$ is a critical point.
\end{proof}

\begin{remark}
	Unfortunately, the function $-F$ has no subdifferential at the boundary of its domain and $S_{d,K}$
	touches this boundary in the anchor set.
	A remedy would be to use instead of $F$ the function
	$$
	F_\varepsilon(A) \coloneqq \sum_{i=1}^N \sqrt{\|y_i\|^2 - \|A^\tT y_i\|^2 + \varepsilon}, \quad \varepsilon > 0.
	$$
	By the proof of Lemma \ref{propF}, we conclude that $-F_\varepsilon$ is convex on an open set which contains $\mathcal{D}$ and therefore also $S_{d,K}$. 
	Thus, accumulation points of the sequence produced by the conditional gradient algorithm are critical points by Theorem~\ref{thm:conv_cond_g}.
	Another idea consists of switching to a function with summands $\varphi(\sqrt{\|y_i\|^2 - \|A^\tT y_i\|^2})$, 
	where $\varphi$ is e.g.~the Huber function as proposed in \cite{ding2006r}.
	%{\color{red} machen irgendwas mit Gewichten}
	This approach is not pursued any further, 
	since we are more interested in finding an algorithm for the original function without an additional parameter.
\end{remark}

\paragraph{Gradient Descent Algorithm on $G_{d,K}$.}
By \eqref{crit_point},
a matrix $A \in S_{d,K} \setminus \mathcal{A}$ 
is a critical point of $E$, resp.~$F$, on $S_{d,K}$ if and only if $ P^{\smallperp}_A C_A A = 0$.
This can be rewritten as
\begin{align}
A (A^\tT C_A A) &= C_A A,\label{eq:fixedpoint_compact}\\
A &= C_A A S_A^{-1} \qquad  \mbox{with} \qquad S_A \coloneqq A^\tT C_A A,\\
A &= A + P^{\smallperp}_A C_A A S_A^{-1},\label{eq:fixed_point_eq}
\end{align}
where $S_A \in \mathbb R^{K,K}$ is assumed to be invertible which is the case 
under the reasonable assumption that 
${\mathcal R}(A)\subset \mathrm{span}(Y)$ and $\mathrm{dim}(\mathrm{span}(Y))\geq K$,
where $Y\coloneqq(y_1 \, \ldots \, y_N)$.

\begin{remark}
Note that $ -\nabla_A E(A)S_A^{-1} = P^{\smallperp}_A C_A A S_A^{-1} \in T_{[A]} G_{d,K} \subset T_A S_{d,K}$.
Let $S_A = Q\Lambda Q^\tT$ be an eigenvalue decomposition with eigenvalues $\lambda_1\geq\ldots\geq \lambda_K>0$ of $S_A$ in the diagonal matrix $\Lambda\in\RR^{K,K}$. 
Plugging this into \eqref{eq:fixedpoint_compact}, multiplying with $Q$ from the right and substituting $\tilde A \coloneqq AQ\in S_{d,K}$ we get
\[
	\tilde A \Lambda = C_{\tilde A}\tilde A,
\]
so that the columns of $\tilde A$ are eigenvectors of 
$C_{\tilde A} = \sum_{i=1}^N \frac{1}{\|P^{\smallperp}_{\tilde A}y_i\|}y_iy_i^\tT$ with eigenvalues $\lambda_k>0$.
\\
Using the same relations in \eqref{eq:fixed_point_eq}, we arrive at
\[
	\tilde A = \tilde A + P^{\smallperp}_{\tilde A} C_{\tilde A} \tilde A \Lambda^{-1} = \tilde A + \nabla_{\tilde A}E(\tilde A)\Lambda^{-1},
\]
which allows the interpretation of $\frac {1}{\lambda_k}$, $k=1,\ldots,K$ as columns-wise step sizes in the gradient descent iteration \eqref{eq:scheme_grad}.
\end{remark}

Together with Lemma \ref{lem:orth_proj}, the fixed point equation \eqref{eq:fixed_point_eq} gives rise to the following  descent scheme on $S_{d,K}$, resp.~$G_{d,K}$:
\begin{align} \label{eq:scheme_grad}
 A^{(r+1)} &\coloneqq \Pi_{S_{d,K}} \left(C_{A^{(r)}}  A^{(r)} S_{A^{(r)}}^{-1} \right)
= C_{A^{(r)}} A^{(r)} S_{A^{(r)}}^{-1} \left( S_{A^{(r)}}^{-1}(A^{(r)})^\tT C_{A^{(r)}}^2 A^{(r)} S_{A^{(r)}}^{-1} \right)^{-\frac 12}.
\end{align}
Note the strong connection of this iterative scheme to the Weiszfeld algorithm \cite{BS2015,We37,NNS2017} 
and majorize-minimize strategies \cite{CIS2011}.

By the following lemma, the gradient descent iteration \eqref{eq:scheme_grad} coincides with those of the conditional gradient algorithm \eqref{cond_grad=ding} on the Grassmannian $G_{d,K}$.

\begin{lemma}\label{lem:scheme_equiv}
	For the same input matrix $A^{(0)}$, 
	the iterates generated by the schemes \eqref{eq:scheme_grad} and \eqref{cond_grad=ding}
	coincide on $G_{d,K}$, i.e., they span the same subspace.
\end{lemma}

\begin{proof}
	Since $C_{AQ} = C_A$ for $Q\in O(K)$, we observe that the matrices produced by the update schemes \eqref{eq:scheme_grad}
	and \eqref{cond_grad=ding} span the same subspace as they differ only by a multiplication with an invertible matrix from the right. 
Since both iterates are in the Stiefel manifold, they can only differ by orthogonal matrix, i.e., 
$\Pi_{S_{d,K}} ( C_{A^{(r)}} A^{(r)} S_{A^{(r)}}^{-1} ) = \Pi_{ S_{d,K}} (C_{A^{(r)}} A^{(r)} ) Q$ for some $Q\in O(K)$.
\end{proof}

The following lemma quantizes the relation from Corollary \ref{cor:cond_grad} that $\{ E(A^{(r)})\}_r$ 
is decreasing.

\begin{lemma}\label{lemma:function_decrease}
	If  $A^{(r)} \notin \mathcal{A}$, then $A^{(r+1)}$ generated by \eqref{eq:scheme_grad} satisfies
	\[
	E\bigl(A^{(r+1)}\bigr) - E\bigl(A^{(r)}\bigr) \leq -  \sum_{i=1}^N \frac{\bigl\|A^{(r+1)}( A^{(r+1)})^\tT y_i - A^{(r)}(A^{(r)})^\tT y_i\bigr\|^2}{2\|P^{\smallperp}_{A^{(r)}}y_i\|}.
	\]
\end{lemma}

\begin{proof}
	In order to shorten notation, we denote $\Aplus = A^{(r+1)}$ and $A = A^{(r)}$. For $x\geq 0, y>0$ it holds $x-y\leq \frac{x^2-y^2}{2y}$ so that
	\begin{align}
	E\bigl(\Aplus\bigr) - E\bigl(A\bigr)
	&= \sum_{i=1}^N \bigl\|P^{\smallperp}_{\Aplus}y_i\bigr\| - \bigl\|P^{\smallperp}_{A}y_i\bigr\| \leq \sum_{i=1}^N \frac{\bigl\|P^{\smallperp}_{\Aplus}y_i\bigr\|^2 - \bigl\|P^{\smallperp}_{A}y_i\bigr\|^2}{2\bigl\|P^{\smallperp}_{A}y_i\bigr\|}\\
	&= \sum_{i=1}^N \frac{\bigl\|\Aplus \Aplus^\tT y_i - y_i\bigr\|^2 - \bigl\|A A^\tT y_i - y_i\bigr\|^2}{2\bigl\|P^{\smallperp}_{A}y_i\bigr\|}.
	\end{align}
	Using $\|u-v\|^2 - \|w-v\|^2 = 2\langle u-w,u-v \rangle - \|u-w\|^2$, this can be rewritten as
	\begin{align}
	E\bigl(\Aplus\bigr) - E\bigl(A\bigr)
	&\leq \sum_{i=1}^N \frac{1}{\|P^{\smallperp}_{A}y_i\|}\left\langle \Aplus \Aplus^\tT y_i - A A^\tT y_i, \Aplus \Aplus^\tT y_i -y_i \right\rangle -  \sum_{i=1}^N \frac{\bigl\| \Aplus \Aplus^\tT y_i - A A^\tT y_i\bigr\|^2}{2\|P^{\smallperp}_{A}y_i\|}\\
	&= \sum_{i=1}^N \frac{\bigl\langle A A^\tT y_i, P^{\smallperp}_{\Aplus}y_i \bigr\rangle}{\|P^{\smallperp}_{A}y_i\|} -  \sum_{i=1}^N \frac{\bigl\|\Aplus \Aplus^\tT y_i - A A^\tT y_i\bigr\|^2}{2\|P^{\smallperp}_{A}y_i\|}.
	\label{eq:wf_proof_mid}
	\end{align}
	The first sum can be simplified to
	\begin{align*}
	\sum_{i=1}^N \frac{\bigl\langle A A^\tT y_i, P^{\smallperp}_{\Aplus}y_i \bigr\rangle}{\|P^{\smallperp}_{A}y_i\|} 
	&=
	\tr\biggl( \sum_{i=1}^N \frac{y_iy_i^\tT}{\|P^{\smallperp}_{A}y_i\|} A A^\tT P^{\smallperp}_{\Aplus} \biggr) = \tr\Big(C_{A} A A^\tT P^{\smallperp}_{\Aplus} \Big) 
	= \tr\Big(C_{A} A A^\tT \bigl(I-\Aplus \Aplus^\tT\bigr) \Big)\\
	&= \tr\left(C_{A} A A^\tT \left( I-C_{A} A \big(A^\tT C_{A}^2 A\big)^{-1}A^\tT C_{A} \right) \right)\\
	&=\tr\left(A^\tT\left( C_{A} A - C_{A} A \big(A^\tT C_{A}^2 A\big)^{-1} A^\tT C_{A}^2 A\right)\right) = 0,
	\end{align*}
	so that
	\[
	E\bigl(\Aplus\bigr) - E\bigl(A\bigr) \leq -  \sum_{i=1}^N \frac{\bigl\|\Aplus \Aplus^\tT y_i - A A^\tT y_i\bigr\|^2}{2\|P^{\smallperp}_{A}y_i\|}.
	\]
\end{proof}

%-----------------------------------------------------
\section{Convergence Analysis}\label{sec:convergence}
%-----------------------------------------------------
So far Corollary \ref{cor:cond_grad} only ensures convergence of a \emph{subsequence} of the iterates to a critical point under a restrictive  assumption on the anchor directions.
In this section, we prove global convergence of the whole sequence of iterates generated by algorithm \eqref{eq:scheme_grad}
on the Stiefel manifold (and thereby on the Grassmannian) under mild assumptions which are summarized at the end of this section.

The important observation is that both $E$ and $F$ are semi-algebraic functions.
Such functions are typical examples of so-called \emph{Kurdyka-\L ojasiewicz} (KL)  
\emph{functions} \cite{ABRS2010,Kur1998,Lo1963}. A function $f\colon \RR^d \to \RR \cup \{+\infty\}$ with Fr\'echet limiting subdifferential $\partial f$, see \cite{mord2006},
is said to have the \emph{Kurdyka–Łojasiewicz (KL) property} at $u^* \in \dom \partial f$ 
if there exist $\eta \in (0, +\infty)$, 
a neighborhood $U$ of $u^*$ and a continuous concave function $\phi\colon [0,\eta) \to \RR_{\geq 0}$ such that
\begin{enumerate}
	\item $\phi(0) = 0$,
	\item $\phi$ is $C^1$ on $(0,\eta)$,
	\item for all $s\in(0,\eta)$ it holds $\phi'(s)>0$,
	\item for all $x\in U\cup [f(u^*)< f <f(u^*) + \eta]$, 
	the Kurdyka–Łojasiewicz inequality  $\phi'(f(u)-f(u^*)) \dist(0,\partial f(u)) \geq 1$ holds true.
\end{enumerate}
A proper, lower semi-continuous (lsc) function which satisfies the KL property at each point of $\dom \partial f$ is called 
\emph{KL-function}.
For KL-functions, the following theorem \cite[Theorem 2.9]{ABSF13} holds true.

\begin{theorem}\label{thm:KL_conv}
	Let $f\colon \RR^d \to \RR \cup \{\infty\}$ be a KL function. 
	Let $\{ u^{(r)}\}_{r \in \N}$ be a sequence  which fulfills the following conditions:
	\begin{enumerate}
		\item[C1)] There exists $K_1>0$ such that $f(u^{(r+1)}) - f(u^{(r)}) \leq -K_1 \Vert u^{(r+1)} - u^{(r)} \Vert^2$ for every $r\in \N$.
		\item[C2)] There exists $K_2>0$ such that for every $r\in \N$ 
		there exists $w_{r+1} \in \partial f(u^{(r+1)})$ with $\Vert w_{r+1} \Vert \leq K_2 \Vert u^{(r+1)} - u^{(r)} \Vert$,
		where $\partial f$ denotes the Fr\'echet limiting subdifferential of $f$ \cite{mord2006}.
		\item[C3)] There exists a convergent subsequence $\{ u^{(r_j)} \}_{j \in \N}$ with limit $\hat u$ and $f(u^{(r_j)}) \to f(\hat u)$.
	\end{enumerate}
	Then the whole sequence $\{ u^{(r)} \}_{r \in \N}$ converges to $\hat u$ 
	and $\hat u$ is a critical point of $f$ in the sense that $0 \in \partial f(\hat u)$.
	Moreover the sequence has finite length, i.e.,
	\[\sum_{r=0}^\infty \bigl\Vert u^{(r+1)} - u^{(r)} \bigr\Vert < \infty.\]
\end{theorem}

Similar arguments as used in the proof of Theorem \ref{thm:KL_conv} lead to the next corollary, see \cite[Corollary 2.7]{ABSF13}.
\begin{corollary}\label{cor:KL_cor}
	Let $f\colon \RR^d \to \RR \cup \{+\infty\}$ be a KL function.
	Denote by $U$, $\eta$ and $\phi$ the objects appearing in the definition of the KL function.
	Let $\delta, \rho > 0$ be such that $B(u^*,\delta)\subset U$ with $\rho \in (0,\delta)$. 
	Consider a finite sequence $u^{(r)}$, $r= 0, \dots,n$, 
	which satisfies the Conditions C1 and C2 of Theorem \ref{thm:KL_conv} and additionally
	\begin{enumerate}
		\item[C4)] $f(u^*) \leq f(u^{(0)}) < f(u^*) + \eta$,
		\item[C5)] $\Vert u^* - u^{(0)} \Vert + 2\sqrt{\frac{f(u^{(0)}) -f(u^*)}{K_1}} + \frac{K_2}{K_1} \phi(f(u^{(0)}) - f(u^*)) \leq \rho$.
	\end{enumerate}
	If for all $r=0,\dots, n$ it holds 
	$$u^{(r)} \in B(u^*,\rho) \quad \Longrightarrow \quad  u^{(r+1)} \in B(u^*,\delta) \; \mathrm{and} \;  f(u^{(r+1)})\geq f(u^*),$$
	then $u^{(r)} \in B(u^*,\rho)$ for all  $r=0,\dots, n+1$.
\end{corollary}

We start our convergence analysis by showing property C1).

\begin{lemma}\label{lem:suff_decrease}
	Assume that $A^{(r)}\notin \mathcal{A}$ for all $r \ge 1$ generated by \eqref{eq:scheme_grad}. 
	Then, there exists $K_1>0$ such that
	\begin{equation}\label{eq:proof_suff_descent}
		E\bigl(A^{(r+1)}\bigr) - E\bigl(A^{(r)}\bigr) \leq -K_1 \bigl\|A^{(r+1)} - A^{(r)}\bigr\|^2.
	\end{equation}
\end{lemma}

\begin{proof}
	In order to shorten notation, we denote $\Aplus = A^{(r+1)}$ and $A = A^{(r)}$.
	By Lemma \ref{lemma:function_decrease}
	and since $\Vert P^{\smallperp}_{A} y_i \Vert \leq \Vert y_i \Vert \leq \max_{i=1,\dots,N} \Vert y_i \Vert =\vcentcolon 1/2C$, 
	we estimate
	\begin{equation}\label{eq:convthmgrad}
	E\bigl(\Aplus\bigr) - E\bigl(A\bigr) \leq - C \sum_{i=1}^N\bigl\| \Aplus \Aplus^\tT y_i - A A^\tT y_i\bigr\|^2.
	\end{equation}
	Next, we want to estimate the sum on the right hand side. To this end, note that
	\begin{align*}
	\bigl\|\Aplus \Aplus^\tT - A A^\tT\bigr\|_2
	&= 
	\max_{\|y\|=1} \bigl\|(\Aplus \Aplus^\tT - A A^\tT)y\bigr\| = \bigl\|(\Aplus \Aplus^\tT - A A^\tT)y \bigr\|,
	\end{align*}
	with some unit vector $y \in \mathcal{R}(Y)$ as $\mathcal{R}(A) \subseteq \mathcal{R}(Y)$. 
	The latter follows directly from the fact that the columns of $C_{A}$ are in $\mathcal{R}(Y)$. 
	%Here, $Y\in \mathbb R^{d,N}$ denotes the matrix with the data points as columns.
	We can choose a basis of $\mathcal{R}(Y)$ from the data points and w.l.o.g., 
	$y = \sum_{i=1}^{N_1} \alpha_i y_i$, 
	where $N_1 = \dim(\mathcal{R}(Y))$.
	Then, setting $Y_{N_1} \coloneqq (y_1 \, \ldots \, y_{N_1})$, the coefficients can be estimated by 
	$|\alpha_i| \leq \alpha^* \coloneqq \|(Y_{N_1}^\tT Y_{N_1})^{-1}Y_{N_1}^\tT\|_\infty$ for $i=1,\ldots,N_1$.
	Setting $\alpha_i = 0$ for all $i>N_1$, we obtain
	\begin{align*}
	\bigl\|\Aplus \Aplus^\tT - A A^\tT\bigr\|_2^2
	&= 
	\Bigl\|(\Aplus \Aplus^\tT - A A^\tT)\sum_{i=1}^{N} \alpha_i y_i\Bigr\|^2
	\leq 
	\biggl(\sum_{i=1}^{N}\bigl\|(\Aplus \Aplus^\tT - A A^\tT) \alpha_i y_i\bigr\|\biggr)^2 \\
	&\leq 
	N (\alpha^*)^2 \sum_{i=1}^{N}\bigl\|\Aplus \Aplus^\tT y_i - A A^\tT y_i\bigr\|^2.
	\end{align*}
	Using the equivalence of Frobenius and spectral norm, \eqref{eq:convthmgrad} now results in the estimate
	\begin{align*}
	E\bigl(\Aplus\bigr) - E\bigl(A\bigr)
	\leq - \frac{C}{2N(\alpha^*)^2} \bigl\|\Aplus \Aplus^\tT - A A^\tT\bigr\|_2^2
	\leq - \tilde C \bigl\| \Aplus \Aplus^\tT - A A^\tT\bigr\|^2.
	\end{align*}
	It remains to show that $\|\Aplus \Aplus^\tT - A A^\tT\|^2 \geq \|\Aplus - A\|^2$. 
	Since $A \in S_{d,K}$, we get
	\begin{align*}
	\bigl\|\Aplus \Aplus^\tT - A A^\tT\bigr\|^2
	= \tr(I_K) - 2\tr\bigl(A^\tT \Aplus \Aplus^\tT A\bigr)  +\tr(I_K)
	= 2\tr\big(I_K -  B \big),
	\end{align*}
	where
	\begin{align*}
	B 
	\coloneqq A^\tT\Aplus \Aplus^\tT A
	= \big(S_{A}^{-1} A^\tT C_{A}^2 A S_{A}^{-1}\big)^{-1}
	= \big(I_K + S_{A}^{-1} A^\tT C_{A} P^{\smallperp}_{A} C_{A} A S_{A}^{-1}\big)^{-1}.
	\end{align*}
	All eigenvalues of $B$ are in $(0,1)$.
	On the other hand, since $A^\tT \tilde A = B^\frac12$, it holds
	\begin{equation}\label{star}
	\bigl\|\Aplus - A\bigr\|^2 = 2 \tr(I_K) - 2\tr\bigl(A^\tT \Aplus\bigr)   = 2\tr \big(I_K - B^{\frac 12}\big).
	\end{equation}
	Finally, this implies with the smallest eigenvalue $\lambda_{\mathrm{min}}\geq 1$ of the matrix $I_K + B^{\frac 12}$ that
	\begin{align*}
	\bigl\|\Aplus \Aplus^\tT - A A^\tT\bigr\|^2
	&= 2\tr \big(I_K - B \big)
	= 2\tr \Bigl( \bigl(I_K + B^{\frac 12}\bigr)\bigl(I_K - B^{\frac 12}\bigr) \Bigr)\\
	&\geq 2\lambda_{\mathrm{min}}\,  \tr\big(I_K - B^{\frac 12} \big)
	\geq 2\tr\big(I_K - B^{\frac 12} \big)\\
	&= \bigl\| \Aplus - A\bigr\|^2.
	\end{align*}
	\end{proof}

\begin{corollary}\label{cor:grad_bound}
Assume that $A^{(r)}\notin \mathcal{A}$ for all $r \ge 1$ generated by \eqref{eq:scheme_grad}.
Then, it holds
\[
	\lim_{r\to\infty} \bigl\|A^{(r+1)} - A^{(r)}\bigr\|=0.
\]
The set of accumulation points of $\{A^{(r)}\}_r$ is compact and connected in $S_{d,K}$. 
Every accumulation point $\hat A$ which is not an anchor point is a critical point of $E$,
i.e.~$\nabla_{\hat A} E(\hat A) = 0$. The same statements hold true for the corresponding equivalence classes in $G_{d,K}$.
\end{corollary}

\begin{proof}
1. Since  $\{E(A^{(r)})\}_r$ is decreasing and bounded below by zero, we know that 
$\lim_{r\to\infty} E(A^{(r)}) = \hat E$ for some $\hat E\geq 0$. 
Multiplying \eqref{eq:proof_suff_descent} by $-1$, summing and taking the limit yields
\[
	E\bigl(A^{(0)}\bigr) - \hat E \geq K_1 \sum_{r=0}^\infty \bigl\|A^{(r+1)} - A^{(r)}\bigr\|^2.
\]
Consequently, the series on the right-hand side converges and $\lim_{r\to\infty} \|A^{(r+1)} - A^{(r)}\|=0$.
Using the estimate
\begin{align}
\bigl\|A^{(r+1)}(A^{(r+1)})^\tT - A^{(r)}(A^{(r)})^\tT\bigr\|_2 
&= \frac12 \bigl\| \big(A^{(r+1)} - A^{(r)} \big) \big( (A^{(r+1)})^\tT +  (A^{(r)})^\tT \big) \\
& \quad + \big(A^{(r+1)} + A^{(r)} \big)\big( (A^{(r+1)})^\tT -  (A^{(r)})^\tT \big) \bigr\|_2\\
&\leq 
C\bigl\|A^{(r+1)} - A^{(r)}\bigr\|, \quad C > 0,\label{eq:normequiv}
\end{align}
the statement also holds on $G_{d,K}$.
\\
2.~By the theorem of Ostrowski \cite[p.~173]{ostrowski1973solutions}, 
it follows that the set of accumulation points of $\{A^{(r)}\}_r$ is compact and connected both in $S_{d,K}$ and $G_{d,K}$. 
\\
3.~Since the sequence $\{A^{(r)}\}_r$ is bounded, it has a convergent subsequence. 
Let $A^{(r_j)}$ be a subsequence converging to a non-anchor point $\hat A$ and $T$ be the update operator in \eqref{eq:scheme_grad}, i.e.,
\[
	T(A^{(r)}) = A^{(r+1)} =\Pi_{S_{d,K}} \left(C_{A^{(r)}}  A^{(r)} S_{A^{(r)}}^{-1} \right) = \Pi_{S_{d,K}} \left(A^{(r)} + P^{\smallperp}_{A^{(r)}}C_{A^{(r)}} A^{(r)} S_{A^{(r)}}^{-1} \right).
\]
Then $T$ is continuous for $A^{(r)}\notin \mathcal{A}$ which we can see as follows: 
The projection $\Pi_{S_{d,K}}$ is continuous on all full rank matrices. For $A\in S_{d,K}$,
\[
	(A + P^{\smallperp}_{A}C_{A} A S_{A}^{-1})^\tT(A + P^{\smallperp}_{A}C_{A} A S_{A}^{-1}) = I_K + S_{A}^{-1} A^\tT C_{A}(P^{\smallperp}_{A})^2C_{A} A S_{A}^{-1}
\]
has only eigenvalues larger than $1$, so that the argument of the projection has full rank. 
Furthermore, $C_A$ (and therefore $S_A$) depends continuously on $A$ except for $A\in\mathcal{A}$.\\
Using the continuity of $T$ outside ${\mathcal A}$,
we have $\lim_{j\to\infty} A^{(r_j+1)} = \lim_{j\to\infty} T(A^{(r_j)}) = T(\hat A)$. 
By continuity of $E$, this implies
\[
E(\hat A) = \lim_{j\to\infty} E\bigl(A^{(r_j)}\bigr) = \hat E = \lim_{j\to\infty} E\bigl(A^{(r_j+1)}\bigr) = E\bigl(T(\hat A)\bigr).
\]
By Corollary \ref{cor:cond_grad},  $E$ is strictly decreasing except for
$\hat A = T(\hat A)$ in which case $\nabla_{\hat A} E(\hat A) = 0$.
\end{proof}

\begin{lemma}\label{lem:grad_bound}
	Assume that the elements of the sequence $\{A^{(r)}\}_r$  are generated by \eqref{eq:scheme_grad}
	and the accumulation points do not belong to the anchor set $\mathcal{A}$.
	Then the sequence fulfills C2) in Theorem \ref{thm:KL_conv}, i.e., there exists $K_2>0$ such that
	\[
		\bigl\|\nabla E(A^{(r+1)})\bigr\| \leq K_2\bigl\|A^{(r+1)} - A^{(r)}\bigr\|.
	\]
\end{lemma} 

\begin{proof}
By the assumptions and Corollary \ref{cor:grad_bound} 
the set of accumulation points has a positive distance $\varepsilon$ from $\mathcal{A}$. 
Since $\lim_{r\to\infty} \|A^{(r+1)} - A^{(r)}\|=0$, we have for $r$ large enough that all lines
$\overline{A^{(r)}A^{(r+1)}}$ connecting $A^{(r)}$ and $A^{(r+1)}$ are in a compact set $\Omega \coloneqq \overline{B(0,R)} \setminus {\mathcal A}_{\varepsilon/2}$ for some $R>0$.
Here $B(0,R)$ denotes the ball centered at $0$ with radius $R$ with respect to the Frobenius norm.
Further, $E$ is smooth on an open set containing $\Omega$ so that the mean value theorem implies
	\begin{align}
		\bigl\|\nabla E(A^{(r+1)}) - \nabla E(A^{(r)})\bigr\|  &\leq C\bigl\|A^{(r+1)} - A^{(r)}\bigr\|.       %\leq \sup_{t \in [0,1]} \|H E\left((1-t) A^{(r)} + tA^{(r+1)}\right) \| \|A^{(r+1)} - A^{(r)}\|\\		
	\end{align}
Hence, we can estimate
	\begin{align}
	\bigl\|\nabla E(A^{(r+1)})\bigr\|
	&\leq \bigl\|\nabla E(A^{(r)})\bigr\| + \bigl\|\nabla E(A^{(r+1)}) - \nabla E(A^{(r)})\bigr\|\\
	&\leq \left( \frac{\bigl\|\nabla E(A^{(r)})\bigr\|}{\bigl\|A^{(r+1)} - A^{(r)}\bigr\|} + C \right)\bigl\|A^{(r+1)} - A^{(r)}\bigr\|.\label{eq:gradbound_final}
	\end{align}
Now, \eqref{star} implies 
	\begin{align*}
	\bigl\|A^{(r+1)} - A^{(r)}\bigr\|^2
	&= 2\tr(I_K) - 2\tr\biggl(\left(I_K + S_{A^{(r)}}^{-1} (A^{(r)})^\tT C_{A^{(r)}} P^{\smallperp}_{A^{(r)}} C_{A^{(r)}} A^{(r)} S_{A^{(r)}}^{-1}\right)^{-\frac 12} \biggr)\\
	&= 2\tr(I_K) - 2\tr\biggl( \left(I_K + S_{A^{(r)}}^{-1} \nabla E(A^{(r)})^\tT \nabla E(A^{(r)}) S_{A^{(r)}}^{-1} \right)^{-\frac 12} \biggr).
	\end{align*}
For the second term, we can use the series expansion of the square root given by
	\begin{align}
	&\tr\biggl( \left(I_K + S_{A^{(r)}}^{-1} \nabla E(A^{(r)})^\tT \nabla E(A^{(r)}) S_{A^{(r)}}^{-1} \right)^{-\frac 12} \biggr)\\
	&= \tr(I_K) -\frac 12\bigl\|\nabla E(A^{(r)}) S_{A^{(r)}}^{-1}\bigr\|^2
	 + \mathcal{O}\left(\bigl\|\nabla E(A^{(r)}) S_{A^{(r)}}^{-1}\bigr\|^4\right).
	\end{align}
Plugging this in yields
	\begin{align*}
	\frac{\bigl\|A^{(r+1)} - A^{(r)}\bigr\|^2}{\bigl\|\nabla E(A^{(r)})\bigr\|^2}
	&=\frac{\bigl\|\nabla E(A^{(r)}) S_{A^{(r)}}^{-1}\bigr\|^2 - \mathcal{O}\bigl(\bigl\|\nabla E(A^{(r)}) S_{A^{(r)}}^{-1}\bigr\|^4\bigr)}{\bigl \|\nabla E(A^{(r)})\bigr\|^2}\\
	&\geq \frac{\bigl\|\nabla E(A^{(r)})\bigr\|^2 \lambda_{\mathrm{min}}\bigl(S_{A^{(r)}}^{-1}\bigr)^2 - \mathcal{O}\bigl(\bigl\|\nabla E(A^{(r)})\bigr\|^4 
	\lambda_{\mathrm{max}}(S_{A^{(r)}}^{-1})^4\bigr)}{\bigl\|\nabla E(A^{(r)})\bigr\|^2}\\
	&= \lambda_{\mathrm{min}}\bigl(S_{A^{(r)}}^{-1}\bigr)^2 - \mathcal{O}\left(\bigl\|\nabla E(A^{(r)})\bigr\|^2 \lambda_{\mathrm{max}}\bigl(S_{A^{(r)}}^{-1}\bigr)^4\right).
	\end{align*}
From Corollary \ref{cor:grad_bound} we know that at each accumulation point which is not an anchor point the gradient of $E$ is zero,
so that $\lim_{r\to\infty} \|\nabla E(A^{(r)})\| = 0$. Furthermore, both $\lambda_{\mathrm{min}}\bigl(S_{A^{(r)}}^{-1}\bigr)$ and $\lambda_{\mathrm{max}}\bigl(S_{A^{(r)}}^{-1}\bigr)$ 
depend continuously on $A^{(r)}$ on the compact set $\Omega$, so that they can be bounded by their positive minima and maxima on $\Omega$, respectively. 
Consequently, $\lim_{r\to\infty} \bigl\|\nabla E(A^{(r)})\bigr\|^2 \lambda_{\mathrm{max}}\bigl(S_{A^{(r)}}^{-1}\bigr)^4 = 0$ and we get
\[
	\frac{\bigl\|A^{(r+1)} - A^{(r)}\bigr\|^2}{\bigl\|\nabla E(A^{(r)})\bigr\|^2} 
	= \lambda_{\mathrm{min}}\bigl(S_{A^{(r)}}^{-1}\bigr)^2 - \mathcal{O}\left(\bigl\|\nabla E(A^{(r)})\bigr\|^2 \lambda_{\mathrm{max}}\bigl(S_{A^{(r)}}^{-1}\bigr)^4\right) > \tilde C
\]
for some $\tilde C>0$ and $r$ large enough. Plugging this into \eqref{eq:gradbound_final}, we get for $r$ large enough
\[
	\bigl\|\nabla E(A^{(r+1)})\bigr\| \leq \Bigl( \frac{1}{\sqrt{\tilde C}} + C \Bigr) \|A^{(r+1)} - A^{(r)}\| .
\]
\end{proof}

\begin{theorem}\label{thm:main}
	Assume that the set of iterates $\{A^{(r)}\}_r$ 
	generated by  \eqref{eq:scheme_grad} is infinite and fulfills
	$A^{(r)}\notin \mathcal{A}$ for all $r\geq 0$. 
	Suppose that there is an accumulation point which is not in $\mathcal{A}$.
	Then the whole sequence $\{A^{(r)}\}_{r}$  converges a critical point.
\end{theorem}

\begin{proof}
	We distinguish two cases.
	\begin{enumerate}
		\item If all accumulation points are non-anchor points, then the assertion follows by Theorem \ref{thm:KL_conv}
		together with
		Lemma \ref{lem:suff_decrease}, Corollary \ref{cor:grad_bound} and Lemma \ref{lem:grad_bound}. 
				
		\item If the set accumulation points consists of both anchor and non-anchor points we will show convergence to a non-anchor point by applying Corollary \ref{cor:KL_cor}. 
		 Let $\hat A$ be an accumulation point which is not in the anchor set $\mathcal A$, i.e., $E_i(\hat A) = \|P^{\smallperp}_{\hat A}y_i\|\neq 0$ for all $i=1,\ldots,N$. 
		 Due to the continuity of $E_i$ we can find a ball $B(\hat A, R)$ around $\hat A$ which has positive distance to all anchor points.
		Next, for all the iterates $A^{(r)} \in B( \hat A, R/2 )$ and  $r$ large enough, C1) and C2) are fulfilled by Lemma \ref{lem:suff_decrease} and  Lemma \ref{lem:grad_bound}.
		By the continuity of $E$ and $\phi$, see also the proof of \cite[Theorem~2.9]{ABSF13}, 
		we can choose a ball $B(\hat A, \delta) \subset B( \hat A, R/2 )\cap U$ 
		(where $U$ is from the definition of the KL property), $\rho \in (0,\delta)$ 
		and a starting iterate $A^{(r_0)} \in B(\hat A, \rho)$ which satisfies C4)  and C5)
		from Corollary~\ref{cor:KL_cor}.
		Since $\lim_{r \to \infty} \Vert A^{(r+1)} - A^{(r)} \Vert = 0$ 
		and $\hat A$ is an accumulation point, we can choose $r_0$ such that
		\[A^{(r)} \in B\bigl(\hat A,\rho\bigr) \implies A^{(r+1)} \in B\bigl(\hat A,\delta\bigr), \; E\bigl(A^{(r+1)}\bigr)\geq E(\hat A)\]
		for all $r \geq r_0$.
		Either all iterates after $A^{(r_0)}$ are in $B(\hat A, \rho)$
		or there is a finite sequence $A^{(r_0)},A^{(r_0+1)} ,\ldots, A^{(r_n)}$ such that
		$A^{(r_{n}+1)}$ is the first element outside $B(\hat A, \rho)$.
		But then, by Corollary~\ref{cor:KL_cor}, also the iterate $A^{(r_n+1)}$ is inside $B(\hat A,\rho)$ and hence all iterates stay in $B(\hat A,\rho)$, which is a contradiction.
		As $\rho$ can be chosen arbitrarily small, the whole sequence converges to the anchor point $\hat A$.
	\end{enumerate}
\end{proof}

We can summarize our results under the condition that no iterate is in the anchor set as follows: If the sequence generated by \eqref{eq:scheme_grad} has at least one accumulation point where $E$ is differentiable, i.e., at least one accumulation is not an anchor point, then it converges to that point on the Stiefel manifold and it is a critical point of $E$. In this case, iteration \eqref{eq:scheme_ding} by Ding et al.~converges on the Grassmannian as it coincides with \eqref{eq:scheme_grad} there. In particular, this implies local convergence near differentiable local minimizers of $E$ which are isolated on the Grassmannian. Only in the case that all accumulation points are anchor points, which means that they form a connected component and all have the same function value, we cannot prove convergence of the whole sequence on the Stiefel or Grassmannian manifold.
In the next section we give partial results for the cases which were not fully treated up to now. We investigate an optimality condition for anchor points and a descent step in non-optimal anchor points.

%-------------------------------------------------
\section{Investigation of Anchor Points}\label{sec:ana_anchor}
%-------------------------------------------------
While a local minimizer of $F$ (and $E$) on $S_{d,K} \cap \mathrm{int}({\mathcal D)}$  
is characterized by the Riemannian gradient being zero, this is not possible for minimizers
lying in the anchor set $\mathcal{A}$, since $E$ is not differentiable and the subdifferential of $-F$  is empty there.
%{\color{red} In experiments with high dimensional data (images) getting stuck in anchor points was not reported in the literature and was also not observed in experiments done by the authors.}

To formulate optimality conditions for matrices in the anchor set, we use the definition of one-sided directional derivatives.
The \emph{one-sided directional derivative} of a function $f\colon \mathbb R^n \rightarrow \mathbb R$, $n \in \mathbb N$, at a point $x \in \mathbb R^n$ in direction $h \in \mathbb R^n$ is defined by
\[
	\Dif f(x;h) \coloneqq \lim_{\alpha\downarrow 0}\frac{f(x+\alpha h) - f(x)}{\alpha}.
\]
The following theorem gives  necessary and sufficient conditions for (strict) local minimizers of (locally Lipschitz) continuous functions on $\RR^n$ 
using the notion of one-sided derivatives, see \cite[Theorems 2.1 \& 3.1]{Ben-Tal1985}.

\begin{theorem}\label{prop:loc_min}
	Let $f\colon \mathbb R^n \rightarrow \mathbb R$ be a function which is one-sided differentiable.
	Then the following holds true:
	\begin{enumerate}
		\item If  $\hat x \in \RR^n$ is a local minimizer of  $f$ on $\RR^n$, then
		$\Dif f(\hat x;h) \ge 0$ for all $h \in \RR^n$. 
		\item If $\Dif f(\hat x;h) >0$ for all $h\in \RR^n \setminus \{0\}$ and $f$ is locally Lipschitz continuous, then $\hat x$ is a strict local minimizer of $f$ on $\RR^n$.
	\end{enumerate}
\end{theorem}
Note that $\Dif f(\hat x;h) \ge 0$ for all $h\in \RR^n \setminus \{0\}$ does not imply that $\hat x$ is a local minimizer of $f$ on $\RR^n$.
The authors of \cite{Ben-Tal1985} gave an example that Lipschitz continuity in the second part  cannot be weakened to just continuity.

The theorem can be applied for complete Riemannian manifolds $\mathcal{M}$ in the following way.
For a function $f \colon \mathcal M \to \RR$, 
the point $\hat x \in \mathcal{M}$ is a local minimizer if and only if $0_{\hat x}$ 
is a local minimizer of $\tilde f_{\hat x} \colon T_{\hat x}\mathcal{M} \to \RR$ with $\tilde f_{\hat x} = f \circ \exp_{\hat x}$, 
where $\exp_{\hat x} \colon T_{\hat x}\mathcal{M} \to \mathcal{M}$ denotes the exponential map at $\hat x$.
The exponential map satisfies 
$\exp_{\hat x}(0_{\hat x}) = \hat x$ and 
$\mathrm{d} \exp_{\hat x}(0_{\hat x}) = \mathrm{Id}$, where $\mathrm{d} F$ denote the differential of a smooth mapping $F$ between manifolds,
see \cite[Section 5.4]{AMS08}.
Hence, local minimizers can be checked with the following directional derivative on manifolds
\[\Dif_{\mathcal M}f(\hat x;h) \coloneqq \Dif\tilde f_{\hat x}(0_{\hat x};h), \quad h \in T_{\hat x}\mathcal{M}.\]
%If $\mathcal M$ is an embedded submanifold of $\RR^n$ and $f$ has a smooth extension to $\RR^n$, this reduces to computing the gradient of $f$.}

Now, we want to apply Theorem~\ref{prop:loc_min} for the locally Lipschitz continuous energy function $E$.
To this end, the norm on $\mathbb R^{n,m}$ is defined by
\[
	\|B\|_{2,1} \coloneqq \sum_{i=1}^m \|b_i\|, \qquad B \coloneqq (b_1 \, \ldots \, b_m).
\]
Then, it is easy to check that the dual norm is given by
\begin{equation} \label{norm_dual}
\|B\|_{2,1}^* = \sup_{\|Z\|_{2,1} = 1} \langle Z,B \rangle_F = \max_{i=1,\ldots,m} \|b_i\| =\vcentcolon \|B\|_{2,\infty}.
\end{equation}
Further, note that for $H \in T_A S_{d,K}$ the exponential map $\exp_{A}$ satisfies
\begin{equation} \label{eq:DirectionalExp}
\od{}{\alpha}(\exp_{A}(\alpha H))\big|_{\alpha=0} = \mathrm{d}  \exp_{A} (0 \, H) \left[ \od{}{\alpha} (\alpha \, H)\big|_{\alpha=0}\right] = H.
\end{equation}
First, the one-sided derivative of $E$ at $A \in S_{d,K}$ in direction $H \in T_A S_{d,K}$ is computed.

\begin{lemma}
	The one-sided derivative of $E$ defined in \eqref{E} on $S_{d,K}$ reads for
	$A \in S_{d,K}$ and $H = A X + A_\perp Z \in T_A S_{d,K}$ as follows
	\begin{equation}\label{one-sided}
		\mathrm{D}_{S_{d,K}} E(A;H) = -\langle Z, A_\perp^\tT C_{A,\mathcal{K}} A \rangle_F +  \| Z A^\tT Y_{\mathcal{K}} \|_{2,1}, 
	\end{equation}
	where $\mathcal{K} \coloneqq \{ k \in \{1,\ldots,N\}: \|P^{\smallperp}_A y_k\| = 0 \}$, $ Y_{\mathcal{K}} \coloneqq (y_k)_{k \in \mathcal{K}}$ and
	\[
		C_{A,\mathcal{K}}\coloneqq \sum_{i=1 \atop i \not \in \mathcal{K}}^N \frac{1}{\|P^{\smallperp}_A y_i\|} y_i y_i^\tT.
	\]
\end{lemma}

%Clearly, if $\mathcal{K}$ is empty, then \eqref{one-sided} simplifies to 
%$DE(A;H) = \langle \nabla E,H \rangle$.

\begin{proof}
	First, we consider $k \in \mathcal{K}$, i.e.~$P^{\smallperp}_A y_k = 0$ and $y_k = A A^\tT y_k$.
	Then, we obtain for $A \in S_{d,K}$ and $H \in T_A S_{d,K}$ that
	\begin{align*}
	\mathrm{D}_{S_{d,K}} E_k (A;H) 
	&= \lim_{\alpha\downarrow 0} \frac{E_k(\exp_A(\alpha H)) - E_k(A)}{\alpha}\\
	&= \biggl\|\lim_{\alpha\downarrow 0} \frac{(I_d-\exp_A(\alpha H)\exp_A(\alpha H)^\tT) y_k}{\alpha}\biggr\|
	\end{align*}
and with $g(B) \coloneqq (I_d - BB^\tT)y_k$ further by the chain rule and \eqref{eq:DirectionalExp},
	\begin{align*}
\mathrm{D}_{S_{d,K}} E_k (A;H) 	
&= 
\left\| \od{}{\alpha}\bigl(g \circ \exp_{A}(\alpha H)\bigr)\big|_{\alpha=0}\right\| 
= 
\left\| \mathrm{d} g \left(\exp_{A}(0 \, H)\right)\bigl[\od{}{\alpha}(\exp_{A}(\alpha H))\big|_{\alpha=0} \bigr] \right\|
\\
&= \left\| \mathrm{d} g (A)[H ] \right\| 
= \bigl\| (A H^\tT + HA^\tT)y_k \bigr\|.
\end{align*}
	Since $H = A X + A_\perp Z$ for some  $X^\tT = -X$ and $y_k \in \mathcal{R}(A)$, this implies further
	\[
		\mathrm{D}_{S_{d,K}} E_k (A;H) = \bigl\|A_\perp Z A^\tT y_k\bigr\| = \bigl\| Z A^\tT y_k\bigr\|.
	\]
	For $k \not \in \mathcal{K}$, the one-sided derivative in direction $H$ is the inner product of $\nabla E_k$ and $H$ so that
	\[
		\mathrm{D}_{S_{d,K}} E(A;H) = -\bigl\langle P^{\smallperp}_A C_{A,\mathcal{K}}  A, H \bigr\rangle + \sum_{k \in \mathcal{K}} \bigl\| Z A^\tT y_k\bigr\|
	\]
	and using the structure of $H$ again, this implies
	\[
		\mathrm{D}_{S_{d,K}} E(A;H) = -\bigl\langle Z, A_\perp^\tT C_{A,\mathcal{K}} A \bigr\rangle +  \bigl\| Z A^\tT Y_{\mathcal{K}}\bigr\|_{2,1}.
	\]
\end{proof}

Under certain conditions, it is possible to formulate a minimality condition also for matrices in the anchor set.

\begin{theorem}\label{cor:condition}
	Let $y_i \in \mathbb R^d$, $i=1,\ldots,N$ and $A \in S_{d,K}$, where $K \le d$.
	Let $\mathcal{K} \coloneqq \{ k \in \{1,\ldots,N\}: \|P^{\smallperp}_A y_k\| = 0 \}$ have cardinality $\kappa \geq 1$.
	Assume that the matrix $Y_{\mathcal{K}} \coloneqq (y_k)_{k \in \mathcal{K}} \in \mathbb R^{d,\kappa}$ 
	has rank $m \leq K$, where the columns are ordered so that the first $m$ are linearly independent 
	and denoted by $Y_m$ and the other ones are their multiples, i.e.,
	$Y_{\mathcal{K}} = (Y_m \, | \, Y_m D)$,
	where $D \in \mathbb R^{m,\kappa - m}$
	is a matrix whose columns contain exactly one nonzero entry.
	Then $A \in S_{d,K}$ is a strict local minimizer of $E$ on  $S_{d,K}$ 
	if and only if the following two conditions are fulfilled
	\begin{align} \label{cond_anchor}
	\bigl\|P^{\smallperp}_A C_{A,\mathcal{K}} (Y_m^\tT Y_m)^{-1} \mathrm{diag}\left( 1_m + |D| \, 1_{\kappa -m} \right)^{-1} \bigr\|_{2,\infty} < 1 \quad \mbox{and} \quad
	P^{\smallperp}_A C_{A,\mathcal{K}} A (A^\tT Y_m)_\perp = 0_{d-K,m},
	\end{align}
	where the absolute value of $D$ has to be understood componentwise, $1_{\kappa -m}$ denotes the vector with $\kappa -m$ entries one,
	and $(A^\tT Y_m)_\perp \in \mathbb R^{K,K-m}$ is any matrix of rank $K-m$ which columns are orthogonal to those of $A^\tT Y_m \in \mathbb R^{K,m}$.
	
	If $Y_{\mathcal{K}}$ contains only vectors which are multiples of $y_1 \in \mathbb R^d$,
	then $A \in S_{d,K}$ is a strict local minimizer of $E$  on $S_{d,K}$ if and only if the following conditions are fulfilled
	\begin{align}
	\bigl\|P^{\smallperp}_A \, C_{A,\mathcal{K}}  \frac{y_1}{\|y_1\|} \bigr\|  <  \|Y_{\mathcal{K}}\|_{2,1} 
	\quad \mbox{and} \quad
	P^{\smallperp}_A  \, C_{A,\mathcal{K}}\, A = P^{\smallperp}_A  \, C_{A,\mathcal{K}}\, \frac{y_1 y_1^\tT}{\|y_1\|^2} \,A. \label{cond2}
	\end{align}
\end{theorem}

\begin{proof}
	By Theorem \ref{prop:loc_min} and \eqref{one-sided}, $A$ is a strict local minimizer of $E$ on  $S_{d,K}$
	if and only if
	$$
	\langle Z, A_\perp^\tT C_{A,\mathcal{K}} A \rangle <  \sum_{k =1}^\kappa \|Z A^\tT y_k\| = \| Z A^\tT Y_{\mathcal{K}}\|_{2,1}
	$$ 
	for all $Z \in \mathbb R^{d-K,K}$. Replacing $Z$ by  
	$Z 
	\begin{pmatrix} 
	Y_m^\tT A\\
	(Y_m^\tT A)_\perp
	\end{pmatrix}
	$,
	where $(Y_m^\tT A)_\perp = (A^\tT Y_m)_\perp^\tT$,
	this is equivalent to
	\[
	\left\langle Z 
	\begin{pmatrix} 
	Y_m^\tT A\\
	(Y_m^\tT A)_\perp
	\end{pmatrix},
	A_\perp^\tT C_{A,\mathcal{K}} A \right\rangle
	= 
	\Bigl\langle Z,
	A_\perp^\tT C_{A,\mathcal{K}} A \bigl(A^\tT Y_m \, | \, (A^\tT Y_m)_\perp\bigr)\Bigr\rangle
	< 
	\Biggl\| Z \begin{pmatrix} Y_m^\tT A A^\tT Y_{\mathcal K}\\ 0_{r-m,\kappa} \end{pmatrix}\Biggr\|_{2,1} 
	\]
	for all $Z \in \mathbb R^{d-K,K}$. 
	Clearly, the condition is fulfilled if and only if
	\[
	\left\langle Z_m, A_\perp^\tT C_{A,\mathcal{K}} A A^\tT Y_m\right\rangle < \bigl\|Z_m Y_m^\tT A A^\tT Y_{\mathcal{K}} \bigr\|_{2,1}
	\quad \mbox{and} \quad 
	A_\perp^\tT C_{A,\mathcal{K}} A (A^\tT Y_m)_\perp = 0_{d-K,m}  
	\]
	for all 
	$Z_m \in \mathbb R^{d-K,m}$. 
	The second part implies that the columns of $C_{A,\mathcal{K}} A (A^\tT Y_m)_\perp$ 
	are in the range of $A$ which gives the second condition in \eqref{cond_anchor}.
	
	Now, $\|P^{\smallperp}_A y_k\| = 0$ implies $A A^\tT y_k = y_k$, $k \in \mathcal{K}$ so that the first condition becomes
	\[
	\left\langle Z_m, A_\perp^\tT C_{A,\mathcal{K}}  Y_m\right\rangle < \bigl\|Z_m Y_m^\tT Y_{\mathcal{K}} \bigr\|_{2,1} .
	\]
	Using $Y_{\mathcal{K}} = (Y_m \, | \, Y_m D)$ and the definition of $\| \cdot \|_{2,1}$,
	the right-hand side can be rewritten as
	\[
	\bigl\|Z_m Y_m^\tT Y_{\mathcal{K}} \bigr\|_{2,1}
	= 
	\bigl\|Z_m \, (Y_m^\tT Y_m) (I_m \, | \, \mathrm{diag}(|D| \, 1_{\kappa -m})  \bigr\|_{2,1}
	\]
	so that the condition can be rewritten as
	\[
	\left\langle Z_m, A_\perp^\tT C_{A,\mathcal{K}}  Y_m\right\rangle \le \bigl\|Z_m  (Y_m^\tT Y_m) (I_m + \mathrm{diag}(|D| \, 1_{\kappa -m})) (Y_m^\tT Y_m) \bigr\|_{2,1}
	\]
	for all $Z_m \in \mathbb R^{d-K,m}$. 
	By \eqref{norm_dual} this is fulfilled if and only if
	\[
	\bigl\| A_\perp^\tT C_{A,\mathcal{K}} Y_m (Y_m^\tT Y_m)^{-1} \mathrm{diag}\left( 1_m +|D| \, 1_{\kappa -m} \right)^{-1} \bigr\|_{2,\infty} < 1.
	\]
	Using $\|P^{\smallperp}_A y\| = \|A_\perp^\tT y\|$ for all $y \in \mathbb R^d$, this gives the assertion \eqref{cond_anchor}.
	
	Assume that the columns of $Y_{\mathcal{K}}$ are multiples of $y_1$.
	Then the first condition of the simplification \eqref{cond2} follows immediately from
	\[(Y_m^\tT Y_m)  \mathrm{diag}\left( 1_m +|D| \, 1_{\kappa -m} \right) =   \sum_{k \in \mathcal{K} } \|y_k\| = \|Y_{\mathcal{K}}\|_{2,1}.\]
	Since for every $x \in \mathbb R^K$, $x \not = 0$ the columns of $I_K - x x^\tT / \|x\|^2$ span the linear space orthogonal to $x$, 
	the second condition  can be deduced using
	$
	P^{\smallperp}_A  \, C_{A,\mathcal{K}}\, A \left(I_K - A^\tT y_1 y_1^\tT A /\|y_1\|^2 \right)=  0_{d,K}
	$
	and $A A^\tT y_1 = y_1$.
\end{proof}

\begin{remark}
	For more general cases than those considered in Theorem \ref{cor:condition} there
	is no simple optimality condition for an anchor point $A$ to be a local minimizer, 
	since basically all possible descent direction $H$ need to be checked in \eqref{one-sided}.
	However, if
	\[
		\bigl\Vert P^{\smallperp}_A C_{A,\mathcal K} A \bigr\Vert^2
		> \sum_{k \in \mathcal K} \bigl\Vert P^{\smallperp}_A C_{A,\mathcal K} y_k\bigr\Vert,
	\]
	then $ H \coloneqq P^{\smallperp}_A C_{A,\mathcal K} A$ is a descent direction since
	\begin{align*}
		\mathrm{D}_{S_{d,K}} E(A;H) &=-\bigl\langle P^{\smallperp}_A C_{A,\mathcal{K}} A, H \bigr\rangle + \sum_{k \in \mathcal{K}} \bigl\| (A H^\tT + HA^\tT)y_k \bigr\|\\
		&= - \bigl\Vert P^{\smallperp}_A C_{A,\mathcal{K}}  A\bigr\Vert^2 + \sum_{k \in \mathcal{K}} \bigl\| P^{\smallperp}_A C_{A,\mathcal K} y_k\bigr\|< 0.
	\end{align*}
	Hence, we can use a line search method to find a next iterate in these anchor points. 
	Note that in comparison to the gradient condition used in \cite[Theorem 4]{MZL19}, ours can be easily verified numerically.
\end{remark}

%-------------------------------------------------------------
\subsection*{Acknowledgments}
	Part of this research was performed while the author was visiting the Institute for Pure and Applied Mathematics (IPAM), which is supported by the National Science Foundation.
	Funding by the German Research Foundation (DFG)  with\-in the Research Training Group 1932,
	project area P3, is gratefully acknowledged. 
	We thank the anonymous reviewer for pointing to the retraction approach in Section 6.

\bibliographystyle{abbrv}
\bibliography{refs_robpca}	

\end{document}